\documentclass[3p]{elsarticle}

\usepackage{amsmath,amsfonts,amssymb, amsthm}

\usepackage{cases}

\usepackage{graphicx}   
\usepackage{graphics,color}
\usepackage{setspace}  
\usepackage{caption}
\usepackage{listings}
\usepackage{tabularx}
\usepackage{xcolor}
\usepackage{indentfirst}
\usepackage{subcaption}

\usepackage{multirow}
\usepackage{overpic}

\usepackage{verbatim}

\usepackage{enumerate}

\usepackage{bm}

\usepackage[font=footnotesize,labelfont=bf]{caption}

\newtheorem{prop}{Proposition}[section]

\newtheorem{rmk}{Remark}[section]

\usepackage[version = 4]{mhchem}

\newcommand\dd{\mathrm{d}}
\newcommand\pp{\partial}

\newcommand\uvec{\bm{u}}
\newcommand\X{\mathbf{X}}
\newcommand\x{\bm{x}}

\newcommand{\fO}{\mathcal{O}}
\newcommand{\fS}{\mathcal{S}}
\newcommand{\rvu}{\bm{u}}

\newcommand{\vu}{\bm{u}}

\newcommand{\vs}{\bm{s}}
\newcommand{\sR}{\mathbb{R}}

\begin{document}

\title{On pattern formation in the thermodynamically-consistent variational Gray-Scott model}

\author[1]{Wenrui Hao}\ead{wxh64@psu.edu}
\affiliation[1]{Department of Mathematics, Pennsylvania State University, University Park, 16802, PA, United States}

\author[2]{Chun Liu}\ead{cliu124@iit.edu}
\affiliation[2]{Department of Applied Mathematics, Illinois Institute of Technology, Chicago, IL 60616, United States}

\author[3]{Yiwei Wang}\ead{yiweiw@ucr.edu}
\affiliation[3]{Department of Mathematics, University of California, Riverside, Riverside, CA 92521, United States}

\author[1]{Yahong Yang}\ead{yxy5498@psu.edu}

\date{}

\begin{abstract}
  In this paper, we explore pattern formation in a four-species variational Gary-Scott model, which includes all reverse reactions and introduces a virtual species to describe the birth-death process in the classical Gray-Scott model. This modification transforms the classical Gray-Scott model into a thermodynamically consistent closed system. The classical two-species Gray-Scott model can be viewed as a subsystem of the variational model in the limiting case when the small parameter $\epsilon$, related to the reaction rate of the reverse reactions, approaches zero. 
  We numerically explore pattern formation in this physically more complete Gray-Scott model in one spatial dimension, using non-uniform steady states of the classical model as initial conditions. By decreasing $\epsilon$, we observed that the stationary pattern in the classical Gray-Scott model can be stabilized as the transient state in the variational model for a significantly small $\epsilon$. 
  Additionally, the variational model admits oscillating and traveling-wave-like pattern for small $\epsilon$. The persistent time of these patterns is on the order of $\fO(\epsilon^{-1})$.
  We also analyze the energy stability of two uniform steady states in the variational Gary-Scott model for fixed $\epsilon$. Although both states are stable in a certain sense, the gradient flow type dynamics of the variational model exhibit a selection effect based on the initial conditions, with pattern formation occurring only if the initial condition does not converge to the boundary steady state, which corresponds to the trivial uniform steady state in the classical Gray-Scott model.
\end{abstract}

\maketitle

\section{Introduction}

Pattern formation is an emerging phenomenon observed across various disciplines, including biology, chemistry, physics, and engineering. Examples range from the stripes on a zebra to the spirals of galaxies \cite{murray2003spatial}. Understanding the mechanisms and principles behind pattern formation has been a longstanding scientific challenge, as it plays a fundamental role in the organization and functionality of natural systems. In biology, pattern formation underpins key processes such as embryonic development, tissue morphogenesis, and cellular differentiation, which are essential for the proper structure and function of living organisms. Mathematical models play a crucial role in understanding pattern formation, offering a powerful tool to unravel its intricate dynamics, uncover fundamental principles, and predict complex behaviors \cite{maini2001mathematical, turing1952chemical, umulis2015role}. During past decades, significant contributions have advanced our understanding of pattern formation, including the molecular underpinnings \cite{maini2006turing}, modeling dispersal in biological systems \cite{othmer1988models}, investigating dynamic pattern formation in cellular networks \cite{othmer1971instability}, and exploring multiscale models of developmental systems  \cite{othmer2006multiscale}. 

Reaction-diffusion equations, rooted in the seminal work of Alan Turing \cite{turing1952chemical}, have been one of the best-known mathematical models used to study pattern formation \cite{kondo2010reaction, meinhardt2000pattern, oster1983mechanical, wu2019synthetic}. 
Arising from Turing instability, reaction-diffusion models can generate a wide variety of spatial patterns. 
Recent advances in the mathematical study of pattern formation include the integration of computational techniques to compute multiple patterns \cite{hao2020spatial, wu2024solution,hao2024companion,zheng2024hompinns,hao2020homotopy,wang2018two}, as well as the application of machine learning methods to analyze complex pattern formation from data and discover novel patterns \cite{huang2022hompinns,zheng2024hompinns}.

One of the most well-known reaction-diffusion systems for pattern formation is the Gray-Scott model, which can generate a diverse range of spatial patterns \cite{pearson1993complex, wang2021bifurcation,chen2023multiple,doelman1997pattern}. Originally proposed by Gray and Scott \cite{gray1984autocatalytic, gray1984chemical}, this model consists of two reacting chemical species in a system of two ordinary differential equations. Later, mechanical effects, i.e., diffusion, was incorporated into the model, leading to reaction-diffusion equations capable of generating complex patterns \cite{pearson1993complex}.
The Gray-Scott model contributes to our understanding of pattern formation observed in various natural and artificial systems \cite{kondo2010reaction}, and has been widely studied during the past decades \cite{doelman1997pattern, hu2012moving, qi2017travelling}. These insights not only deepen our understanding of nonlinear dynamics and self-organization but also have practical implications in fields such as chemistry, materials science, and biology \cite{nakao2010collective}. The Gray-Scott model is a widely used prototype for studying pattern formation, providing a critical link between theoretical modeling and experimental observations across various disciplines \cite{li2017turing}. Similar models for pattern formation include the Belousov-Zhabotinsky reaction \cite{turner2009simple}, the Schnakenberg model \cite{schnakenberg1979simple}, the Gierer-Meinhardt model \cite{gierer1972theory}, and the Brusselator model \cite{gambino2013turing}, which also exhibit rich dynamics and are commonly used in mathematical biology and related fields.

The classical Gray-Scott model is built by coupling linear diffusion equations with the reaction kinetics of two irreversible reactions, described by the law of mass action. More precisely, the Gray-Scott model considers the following two irreversible reactions
\begin{equation}\label{KCR_GS_VS}
  \begin{aligned}
    & \ce{U + 2 V -> 3 V}, \quad \ce{V -> P}\ , 
   \end{aligned}
\end{equation}
where $V$ is an activator, $U$ is a substrate, and $P$ is an inert product. To maintain the system out of equilibrium, both $U$ and $V$ are removed by the feed process \cite{pearson1993complex}. Consequently, 
The reaction-diffusion equation of $u$ and $v$ are given by \cite{pearson1993complex}
    \begin{equation}
      \begin{cases}
       & \dfrac{\pp u}{\pp t} = D_u \Delta u - u v^2 + f(1 - u) \\
       & \\
       & \dfrac{\pp v}{\pp t} = D_v \Delta v + u v^2 - k v - f v \ , \\ 
       \end{cases}\label{classical}
    \end{equation}
subject to certain boundary and initial conditions.
Here, $u$ and $v$ denote the concentration of $U$ and $V$, the dimensionless reaction rate for the first reaction is set to be $1$, $k$ is the dimensionless rate constant of the second reaction, $f$ represents the dimensionless feed rate, and $D_u$ and $D_v$ are diffusion coefficients.  A surprising variety of spatiotemporal patterns can emerge from the reaction-diffusion equation (\ref{classical}) under specific parameter values and initial conditions \cite{pearson1993complex}. 

From a modeling perspective, the classical Gray-Scott model is not thermodynamically consistent, meaning it may not satisfy the first and second laws of thermodynamics. In biological systems, where energy input, dissipation, and conversion govern self-organization, thermodynamic consistency is crucial for developing mechanistically accurate and predictive models that faithfully capture the underlying physical principles. As a result, the thermodynamic basis for pattern formation in the Gray-Scott model is unclear. For instance, it remains uncertain whether these patterns represent transient phenomena or non-equilibrium steady states, what is the energetic cost of the complex pattern in the models.

In a recent work \cite{liang2022reversible}, variational, reversible Gray-Scott models are developed based on an energetic variational approach (EnVarA) \cite{giga2017variational, wang2022some}. The model revises the classical Gray-Scott model into a thermodynamically consistent form by including all reverse reactions in the model. Additionally, a virtual species is introduced to account for the influx and efflux of $u$ and $v$, effectively transforming the open system into a subsystem of a larger, closed system.
It is proved in \cite{liang2022reversible} that in the short-time dynamic process, when $\epsilon$, the small constant related to reaction rates in the reverse part, goes to zero, the solution to the variational Gray-Scott models will converge to that of the classical Gray-Scott models. Additionally, numerical simulations in a recent study \cite{fu2023high} demonstrate pattern formation can occur in the variational model.

The variational Gray-Scott model not only restores thermodynamic consistency of the classical model but also provides a more realistic framework for studying biological and chemical pattern formation. Many natural processes, such as biochemical reactions and cellular transport mechanisms, involve reversible dynamics that adhere to thermodynamic laws. By capturing these aspects, the variational model allows for a deeper understanding of pattern formation in systems where classical models fall short. Moreover, the techniques and insights developed through this study can be extended to other reaction-diffusion systems, offering a roadmap for analyzing complex patterns in fields ranging from developmental biology to materials science.

For the classical Gray-Scott model, many nontrivial steady states on the domain $\Omega = (0,1)$ with no-flux boundary conditions have been computed in \cite{hao2020spatial}. As a first step in exploring pattern formation in the thermodynamically consistent variational Gray-Scott model for various values of $\epsilon$, we use these steady states as initial conditions, expecting they may serve as suitable starting points for observing potential pattern formation in the variational model.
Numerical experiments indicate that for relatively large values of $\epsilon$, all initial conditions quickly converge to a uniform steady state. However, for sufficiently small $\epsilon$, the variational model exhibits rich transient dynamics: stationary patterns may persist for long periods, and both oscillatory and traveling-wave-like behaviors can emerge. We also analyze the stability of two uniform steady states in the variational Gray-Scott model for fixed $\epsilon$ to provide some theoretical insights to the simulation results. Although both uniform states are stable in a certain
sense, pattern formation can only occur if the initial condition does not converge to the boundary steady state, in which $(u, v) = (1, 0)$.

The rest of the paper is organized as follows. In Section 2, we derive the variational Gray-Scott model using the energetic variational approach. Section 3 is devoted to analyzing the energy stability of two uniform steady states. In Section 4, we investigate pattern formation in one dimension for different values of $\epsilon$, using the steady states from \cite{hao2020spatial} as initial conditions. Finally, in Section 5, we examine how the persistence time of patterns depends on the parameter $\epsilon$. A conclusion remark is given in Section 6.

\section{Variational Gray-Scott Model}

In this section, we briefly review the derivation of the thermodynamically consistent variational Gray-Scott model, proposed in \cite{liang2022reversible}, by an energetic variational approach \cite{giga2017variational, wang2022some}.

Motivated by non-equilibrium thermodynamics, particularly the celebrated works of Lord Rayleigh \cite{strutt1871some} and Onsager \cite{onsager1931reciprocal, onsager1931reciprocal2},
the EnVarA has been a powerful tool of building thermodynamically consistent models in physics, chemical engineering, and biology \cite{giga2017variational, wang2020field}. The key idea of EnVarA is to describe an isothermal and mechanically isolated system by its energy and the rate of energy dissipation over time, along with kinematic (transport) assumptions on the employed variables. More precisely, according to the first and second laws of thermodynamics \cite{giga2017variational, ericksen1992introduction}, an isothermal and closed system possesses an energy-dissipation law
\begin{equation}
\frac{\dd}{\dd t} E^{\text{total}}(t) = - \triangle(t) \leq 0 \ .
\end{equation}
Here, $E^{\text{total}}$ is the total energy, which is the sum of the Helmholtz free energy $\mathcal{F}$ and the kinetic energy $\mathcal{K}$; $\triangle (t) \geq 0$ stands for the rate of energy dissipation, which equals to the rate of entropy production in this case. 
Once these quantities are specified, a thermodynamically consistent model can be derived by combining the Least Action Principle (LAP) and the Maximum Dissipation Principle (MDP) \cite{giga2017variational}.  More specifically, for the energy part, one can employ the LAP, taking variation of the action functional $\mathcal{A}(\x) = \int_0^T \left( \mathcal{K} - \mathcal{F} \right) \dd t$ with respect to $\x$ (the trajectory in Lagrangian coordinates) \cite{giga2017variational, arnol2013mathematical},  to derive the conservative force, i.e.,
$\delta \mathcal{A} =  \int_{0}^T \int_{\Omega} ({\rm force}_{\text{iner}} - {\rm force}_{\text{conv}})\cdot \delta \x  ~\dd \x \dd t.$ For the dissipation part, one can apply the MDP, taking the variation of the Onsager dissipation functional $\mathcal{D}$ with respect to the ``rate'' $\x_t$, to derive the dissipative force, i.e., $\delta \mathcal{D}  = \int_{\Omega} {\rm force}_{\text{diss}} \cdot \delta \x_t~ \dd \x$, where the dissipation functional $\mathcal{D} = \frac{1}{2} \triangle$ in the linear response regime \cite{onsager1931reciprocal}. Consequently, the force balance condition 
results in
\begin{equation}\label{FB}
\frac{\delta \mathcal{A}}{\delta \x} = \frac{\delta \mathcal{D}}{\delta \x_t},
\end{equation}
which is the dynamics of the system. In the case that $\mathcal{K} = 0$, $\frac{\delta \mathcal{A}}{\delta \x} = - \frac{\delta \mathcal{F}}{\delta \x}$, then the dynamics can be written as $$\frac{\delta \mathcal{D}}{\delta \x_t} = - \frac{\delta \mathcal{F}}{\delta \x},$$ which is a generalized gradient flow. For these systems, the free energy determines the equilibrium of the system and the rate of energy dissipation determines the dynamics.

The EnVarA is originally developed for mechanical systems,  and $\x$ should be understood as the flow map. Recent work \cite{wang2020field} extends this variational principle to reaction kinetics by computing the variation with respect to a reaction trajectory  $R$, analogous to the flow map, and its time derivative $\pp_t R$, representing the reaction rate, to derive reaction kinetics. One key advantage of using EnVarA to model complex chemo-mechanical systems in biology is its ability to provide a unified framework for both mechanics and chemistry. All multiscale energetic couplings and competing effects are naturally incorporated through the choice of the energy-dissipation law.

\subsection{Derivation of the Variational Gray-Scott Models}

Since we are only interested in the concentration of the substrate $U$ and the activator $V$, we can combine the removal process $V$ and the process of generating $P$, an inert product, as one process for mathematical simplicity. Therefore, the chemical reactions described in the classical Gray-Scott model (\ref{classical}) can be represented as follows:
\begin{equation}\label{KCR_GS_i}
  \begin{aligned}
    & \ce{U + 2 V ->[ 1] 3 V}, \quad \ce{V ->[k+f] P}, \quad \ce{U ->[f] $\emptyset$}, \quad \ce{$\emptyset$ ->[f] U}
   \end{aligned}
\end{equation} 
In the reaction network (\ref{KCR_GS_i}), all chemical reactions are irreversible. In addition, it involves the birth and death of $U$, making the overall system an open system.

To formulate a thermodynamically consistent variational Gray-Scott model, we consider the following reversible chemical reactions:
\begin{equation}\label{KCR_GS_Re}
  \ce{U + 2 V <=>[1][\epsilon_1] 3V}, \quad \ce{V <=>[k$+$f][\epsilon_2] P}, \quad \ce{U <=>[f][\epsilon_3] Y}\ ,
\end{equation}
where $\epsilon_i$ are small parameters, $Y$ is a virtual species added to model the birth-death process of $U$. Clearly, if $\epsilon_1 \rightarrow 0$ and $\epsilon_2 \rightarrow 0$, the first two reactions revert to those in (\ref{KCR_GS_i}). 
The introduction of the virtual species $Y$, whose concentration is large (of the order $O(1/\epsilon_3)$), allows us to treat the open system, situated in a sustained environment with influx and efflux, as a subsystem in a larger, closed ``universe'' \cite{ge2013dissipation}. 
A similar approach is used in \cite{falasco2018information}.
In general, the small parameters $\epsilon_i$ are different and may not be in the same order.  
In particular, $\epsilon_3$ not only represents the reaction rate of the reverse reaction but also describes the concentration scale of the virtual species $Y$.
In the current study, we assume $\epsilon_1 = \epsilon_2 = \epsilon_3  = \epsilon$. We will explore more general cases in future work.


We denote the concentrations of species $P$ and $Y$ by $p$ and $y/\epsilon$ respectively. 
To derive the variational Gray-Scott model, we first notice that the concentrations satisfy the following kinematics
  \begin{equation}
    \begin{cases}
      & u_t + \nabla \cdot (u \rvu_u) = - \pp_t R_1 - \pp_t R_3 \\
      & v_t + \nabla \cdot (v \rvu_v) =  \pp_t R_1 - \pp_t R_2 \\
      & p_t = \pp_t R_2 \\
      & y_t = \epsilon \pp_t R_3, \\
    \end{cases}\label{law}
  \end{equation}
where $R_i$ is the reaction trajectory of each reaction in (\ref{KCR_GS_Re}), and $\rvu_u$ and $\rvu_v$ are effective velocity induced by the diffusion process. The reaction trajectories $R_i(t)$ account for the number of forward reaction has happened for $i$-th reaction by time $t$ and may take negative values \cite{wang2020field}.
Here, we disregard the diffusion effects on $P$ and $Y$, as the slow diffusion of P and Y minimally impacts the dynamics of U and V, which are our main focus.  
The boundary condition for all species is the non-flux  boundary condition, given by
\begin{equation}\label{BC_no_flux}
  \frac{\partial u}{\partial \mathbf{\nu}} = \frac{\partial v}{\partial \mathbf{\nu}} =  \frac{\partial p}{\partial \mathbf{\nu}} =  \frac{\partial y}{\partial \mathbf{\nu}} = 0, \
\end{equation}
which ensures the boundary term vanishes in deriving the force balance equation by the EnVarA.
It is important to note that
\begin{equation}\label{conserved}
  \frac{\dd}{\dd t} \int_{\Omega} u + v + p + \frac{y}{\epsilon} ~ \dd x = 0 \ ,
\end{equation}
for the kinematics (\ref{law}) along with the non-flux boundary condition (\ref{BC_no_flux}). The conservation property (\ref{conserved}) plays an important role in studying the steady state of the variational Gray-Scott model.

Following in general framework of modeling reaction-dissipation systems \cite{wang2020field}, the reaction-diffusion system can be modeled through the energy-dissipation law
\begin{equation}\label{ED_GS}
 \frac{\dd}{\dd t} \mathcal{F}(u, v, p, y) = - \triangle_{\rm mech} - \triangle_{\rm chem}\ .
\end{equation}
Here, $\mathcal{F}(u, v, p, y)$ is the free energy of the system, $\triangle_{\rm mech}$ and $\triangle_{\rm cheme}$ are the rate of energy-dissipation due the mechanical (diffusion) and chemical (reaction) parts respectively. The free energy $\mathcal{F}(u, v, p, y)$ is taken as
\begin{equation}\label{free_energy_GS}
    \mathcal{F} = \int u (\ln u - 1) + u \sigma_u + v (\ln v - 1) + v \sigma_v + p (\ln p - 1) + p \sigma_p +  \widetilde{y} (\ln \widetilde{y} - 1) + \widetilde{y}\sigma_{\widetilde{y}} \dd x\ .
\end{equation}
Here, $\widetilde{y} = \frac{y}{\epsilon}$ represents the concentration of $Y$, $\sigma_{u}$, $\sigma_{v}$, $\sigma_{p}$, $\sigma_{\widetilde{y}}$  denote internal energy of each species, which together determine the equilibrium of the system.
Let $(u^s_{+}, v^s_{+}, p^{s}_{+}, \widetilde{y}^{s}_{+}) \in \mathbb{R}^4_{+}$ be a positive equilibrium of the chemical reaction system (\ref{KCR_GS_Re}),  then $\sigma_{u}$, $\sigma_{v}$, $\sigma_{p}$, $\sigma_{\widetilde{y}}$  satisfies
\begin{equation}\label{sigma}
    \begin{cases}
      & \ln u^{s}_+ + \sigma_u = \ln v^{s}_+ + \sigma_v \\ 
      & \ln p^{s}_+ + \sigma_p = \ln v^{s}_+ + \sigma_v \\
      & \ln u^{s}_+ + \sigma_u = \ln \widetilde{y}^{s}_+ + \sigma_{\widetilde{y}}.
    \end{cases}
\end{equation}
Since $(u^s_{+}, v^s_{+}, p^{s}_{+}, \widetilde{y}^{s}_{+})$ satisfies
\begin{equation}
    u^s_{+} (v^s_{+})^2 = \epsilon (v^s_{+})^3, \quad (k+f) v_s^+ = \epsilon p^s_+, \quad f u_s^+ = \epsilon \widetilde{y}^{s}_+ \ ,
\end{equation}
we have
\begin{equation}\label{sigma_1}
  \sigma_{v} -  \sigma_u = \ln \epsilon, \quad \sigma_p  - \sigma_v = \ln \epsilon - \ln (k+f), \quad  \sigma_{\widetilde{y}} - \sigma_u = \ln \epsilon - \ln f \ .
\end{equation}
We can solve for $\sigma_i$ ($i = u, v, p, \widetilde{y}$) in terms of $k$, $f$ and $\epsilon$.
Notice in Eq.~\eqref{sigma_1}, that there are four variables but only three equations, resulting in an overparameterized case. 
The analysis for the differential internal energy $\sigma_i$ follows a similar pattern. In this paper, we take 
\begin{equation}
\sigma_u = 0, \quad \sigma_v = \ln \epsilon, \quad \sigma_p = 2 \ln \epsilon-\ln (k+f), \quad \sigma_{\widetilde{y}} = \ln \epsilon - \ln f.
\end{equation}

Next, we impose the rate of energy dissipation of the system, given by
\begin{equation}
  \triangle_{\rm mech} = \int  \int \frac{u}{D_v} |\rvu_u|^2 + \frac{v}{D_v} |\rvu_v|^2 \dd x
\end{equation}
and
\begin{equation}
\triangle_{\rm chem} = \int \pp_t R_1 \ln \left( \frac{\pp_t R_1}{\epsilon_1 v^3} + 1 \right) + \pp_t R_2 \ln \left( \frac{\pp_t R_2}{\epsilon p} + 1 \right) + \pp_t R_3 \ln \left( \frac{\pp_t R_3}{y} - 1 \right)  \dd x 
\end{equation}

To derive the dynamics from the energy-dissipation, we apply the EnVarA to the mechanical and chemical parts respectively, which leads to $\rvu_{\alpha}$ ($\alpha = u, v$) and $R_i$ ($i = 1, 2, 3$) such that the energy-dissipation law (\ref{ED_GS}) holds. In order to get the equation of $\rvu_u$, we first define the flow map $\x_u(\X, t)$ that is associated with $\rvu_u$ by the ODE
\[  \frac{\dd}{\dd t} \x_u(\X, t) = \uvec_u(\x_u(\X, t), t), \quad x_u(\X, 0) = \X\ . \]
If we consider only the diffusion of $u$, $u(\x, t)$ is determined by the flow map $\x_u(\X, t)$, making $\mathcal{F}$ as a functional of $\x_u(\X, t)$. Applying LAP and MDP with respect to $\x_u$ and $\rvu_u$ leads to the force balance equation (see \cite{liang2022reversible} for the detailed calculations) 
\begin{equation}\label{FB_1}
u \rvu_{u} = - D_u u \nabla \mu_u, \quad  \mu_u = \frac{\delta \mathcal{F}}{\delta u} = \ln u  + \sigma_{u}
\end{equation}
where $\mu_u$ is known as the chemical potential for the species $u$.
Similarly, we can derive that 
\begin{equation}\label{FB_2}
v \rvu_v = - D_v v \nabla \mu_v, \quad \quad  \mu_v = \frac{\delta \mathcal{F}}{\delta v} = \ln v  + \sigma_{v}
\end{equation}
For the chemical part, we need to perform the EnVarA with respect to $R_i$ and $\pp_t R_i$, which leads to
\begin{equation}\label{FB_3}
\ln \left( \frac{\pp_t R_1}{\epsilon v^3} + 1 \right) = - (\mu_v - \mu_u), \quad \ln \left( \frac{\pp_t R_2}{\epsilon p} + 1 \right) = -(\mu_p - \mu_v), \quad \ln \left( \frac{\pp_t R_3}{y} - 1 \right) = - (\mu_{\tilde{y}} - \mu_u)\
\end{equation}
Recall that $\sigma_{u}$, $\sigma_{v}$, $\sigma_{p}$, $\sigma_{\widetilde{y}}$  satisfies (\ref{sigma_1}), we can show $\pp_t R_i$ satisfies the law of mass action. Indeed, from (\ref{FB_3}), we know that
\begin{equation}
\pp_t R_1 = \epsilon v^3 \exp( - (\ln v  + \sigma_v - \ln u - \sigma_u ) ) - \epsilon v^3 = u v^2 - \epsilon v^3
\end{equation}
The calculations for $\pp_t R_2$ and $\pp_t R_3$ are similar.

Combing the force balance equations (\ref{FB_1}), (\ref{FB_2}) and  (\ref{FB_3}) with the kinematics (\ref{law}), we end up with the following reversible variational Gray-Scott model
\begin{equation}
  \begin{cases}
    & u_t = D_u \Delta u - ( u v^2 - \epsilon v^3) - f u + y \\
    & v_t = D_v \Delta v + ( u v^2 - \epsilon v^3) - ((k+f) v - \epsilon p) \\
    & p_t =   (k+f) v -  \epsilon p \\
    & y_t = \epsilon (f u - y). \\ 
  \end{cases} 
  \label{modGS}
\end{equation}

\begin{rmk}
Although the Gray-Scott model is not directly associated with any specific biological system, the same approach can be applied to restore thermodynamic consistency in more realistic models of pattern formation in various biological systems, as these models are typically of the reaction-diffusion type \cite{wu2019synthetic}. As another representative example, consider the celebrated Schnakenberg model \cite{maini1997spatial, schnakenberg1979simple}, given by
\begin{equation}
\begin{cases}
u_t = D_u \Delta u + k_2 a - k_1 u + k_3 u^2 v, \\
v_t = D_v \Delta v + k_4 b - k_3 u^2 v,
\end{cases}
\end{equation}
which models the sequence of reactions
\[
X \ce{<=>[k_1][k_2]} A, \quad 2X + Y \ce{->[k_3]} 3X, \quad B \ce{->[k_4]} Y,
\]
where $u$ and $v$ denote the concentrations of species $X$ and $Y$, respectively, and $a$ and $b$ are the concentrations of $A$ and $B$, assumed to be constant. A thermodynamically consistent version of the Schnakenberg model can be constructed by restoring the reversible parts of the latter two reactions with small reaction rate $\epsilon$ and formulating the full system in terms of $a$, $b$, $u$, and $v$. The resulting reaction-diffusion system can be derived from the energy-dissipation law
\begin{equation*}
\begin{aligned}
 & \frac{\dd}{\dd t}\int u (\ln u - 1) + u \sigma_u + v (\ln v - 1) + v \sigma_v + a (\ln a - 1) + a \sigma_a +  b (\ln b - 1) + b \sigma_b \dd x  \\
 & - \int \frac{u}{D_u} |{\bm u}_u| + \frac{v}{D_v} |{\bm u}_v| + \pp_t R_1 \ln \left( \frac{\pp_t R_1}{k_2 a} + 1 \right) + \pp_t R_2 \ln \left( \frac{\pp_t R_2}{\epsilon u^3} + 1 \right) + \pp_t R_3 \ln \left( \frac{\pp_t R_3}{\epsilon v} + 1 \right)
\end{aligned}
\end{equation*}
where $R_i$ are reaction trajectories for the three involved reactions, and ${\bm u}_{\alpha} (\alpha = u, v)$ are average velocities due to the diffusion of $u$ and $v$. The internal energies can be taken as $\sigma_a = 0$, $\sigma_u = \ln k_2 - \ln k_1$, $\sigma_v = \ln \epsilon + \ln k_2 - \ln k_1 - \ln k_3$ and $\sigma_b = 2 \ln \epsilon + \ln k_2 - \ln k_1 - \ln k_3 - \ln k_4$.
\end{rmk}


\subsection{Formal limit of the Variational Gray-Scott Model}

In this subsection, we show the first two equations in the variational Gray-Scott model (\ref{modGS}) can be reduced to the classical Gray-Scott model (\ref{classical}) when $\epsilon \rightarrow 0$.

Assume the initial concentrations of U, V, P, and Y as $u_0$, $v_0$, $p_0$, and $y_0/\epsilon$, where $y_0=f$. When $\epsilon$ is small, we notice that $y(t)\approx y_0 = f$ is a nearly constant function since $y_t\approx 0$. Thus, as $\epsilon$ approaches zero, the first two equations in the variational Gray-Scott model formally converge to the classical Gray-Scott models, as we drop all terms with $\epsilon$ and replace $y$ by $f$ in the $u$, $v$ equation. The above argument can be made more rigorously \cite{liang2022reversible}. Since both the equations of $p$ and $y$ are linear,  we have
\begin{equation}
\begin{aligned}
 p = p_0 e^{-\epsilon t} + e^{-\epsilon t} (k+f) \int_{0}^t v(t) e^{\epsilon t} \dd t \hbox{~~and~~}
 y = y_0 e^{-\epsilon t} + \epsilon f e^{-\epsilon t} \int_{0}^t u(t) e^{\epsilon t} \dd t .
\end{aligned}\label{eq1}
\end{equation}
By plugging Eq. (\ref{eq1}) into Eq. (\ref{modGS}), we rewrite the equations of $u$ and $v$ as follows:
\begin{equation}
\begin{cases}
& u_t = D_u \Delta u - ( u v^2 - \epsilon v^3) - f u + y_0 e^{-\epsilon t} + \epsilon f e^{-\epsilon t} \int_{0}^t u(t) e^{\epsilon t} \dd t \\
& v_t = D_v \Delta v + ( u v^2 - \epsilon v^3) - (k+f) v + \epsilon p_0 e^{-\epsilon t} + \epsilon e^{-\epsilon t} (k+f) \int_{0}^t v(t) e^{\epsilon t} \dd t .\\
\end{cases}\label{Eq8}
\end{equation}
If $\max\{ \| u \|_{L^{\infty}}, \| v \|_{L^{\infty}}\}\le M$ for all $t \in [0,\infty)$, where $M$ is an order one constant with respect to $\epsilon$, then the following limits hold:
\[ \epsilon v^3,\epsilon f e^{-\epsilon t} \int_{0}^t u(t) e^{\epsilon t} \, dt, \epsilon y_0 e^{-\epsilon t}, \epsilon e^{-\epsilon t} (k+f) \int_{0}^t v(t) e^{\epsilon t} \, dt \to 0, \]
and $y_0 e^{-\epsilon t} \to y_0$ as $\epsilon \to 0$. Hence, formally, the variational Gray-Scott model (\ref{Eq8}) converges to the classical Gray-Scott model
\begin{equation}
\begin{cases}
& u_t = D_u \Delta u -  u v^2 - f u + f \\
& v_t = D_v \Delta v +  u v^2 - (k+f) v .\\
\end{cases}
\end{equation}
when $\epsilon \rightarrow 0$. We emphasize that it might be difficult to prove the assumption  $\max\{ \| u \|_{L^{\infty}}, \| v \|_{L^{\infty}}\}\le M$. So, the above limit is formal. However, the numerical simulations in the next section show this assumption holds for all tested initial conditions.

It is worth emphasizing that the formal convergence does not imply that the performance of the variational Gray-Scott models and the classical model is identical, as the system may exhibit singularity with respect to $\epsilon$. In the following sections, we will study the dynamical behavior of the variational Gray-Scott model for various $\epsilon$.


\section{Uniform steady states in the Variational Gray-Scott Model and their energy stability}

Unlike the classical Gray-Scott model (when $\epsilon = 0$), which has only one uniform steady state $(u, v) = (1, 0)$, for a given $\epsilon > 0$, the variational system admits two uniform steady states. More precisely, for a given initial conditions $(u_0, v_0, p_0, y_0)$ and the domain $\Omega = (0, 1)$, we define the total mass of all species $C$ as 
\begin{equation}\label{Def_C}
    C =  \int_{\Omega} u_0 + v_0 + p_0 + \frac{1}{\epsilon} y_0 ~\dd x\ ,
\end{equation}
where $y_0 = f$. 
Assume that $(u_s, v_s, p_s, y_s)$ is spatially homogeneous steady-states of the variational Gray-Scott model (\ref{modGS}), then $(u_s, v_s, p_s, y_s)$ satisfies
\begin{equation}\label{cee}
    \begin{cases}
      & f u^{s} = y^{s} \\ 
      & u^{s} (v^s)^2 = \epsilon (v^{s})^3 \\
      &(k+f) v^s = \epsilon p^{s} \\
    \end{cases}
\end{equation}
subject to the constraint:
\begin{equation}\label{contraine}
    u^s + v^s + p^s + \frac{1}{\epsilon} y^s = C / |\Omega| = C\ ,
\end{equation}
Here, $|\Omega| = 1$ is the size of the domain. Solving (\ref{cee}) with the constraint (\ref{contraine}), we obtain two spatially homogeneous steady-states: one is a {\bf boundary steady-state} given by
\begin{equation}
\vu_1:=(u^s_1, v^s_1, p^s_1, y^s_1) := \left(\frac{\epsilon}{\epsilon+f} C,~ 0,~ 0, ~\frac{f\epsilon}{(\epsilon+f)} C\right),
\end{equation}
and the other is an {\bf interior steady-state} expressed as
\begin{equation}
\vu_2:=(u^s_2, v^s_2, p^s_2, y^s_2) := \left(\frac{\epsilon}{\lambda} C, ~ \frac{1}{\lambda} C, ~ \frac{(k+f)}{\epsilon\lambda} C, ~ \frac{f\epsilon}{\lambda(\epsilon)} C\right),
\end{equation}
where 
\[\lambda:=\epsilon+1+\frac{k+f}{\epsilon}+f.\] Here, the term interior means the steady state is in the interior of the stoichiometric compatibility
class \cite{anderson2015lyapunov}, defined by
\begin{equation*}
 \{ (u, v, p,  y) \in \mathbb{R}_+^4 ~|~  u = u_0 - R_1 - R_3, v = v_0 + R_1 - R_2, p = p_0 + R_2, y = y_0 + \epsilon R_3,  (R_1, R_2, R_3) \in \mathbb{R}^3 \}\ .
\end{equation*}
In contrast, the term boundary refers to a steady state that lies on the boundary of this compatibility class.

Notice that $\lambda=\mathcal{O}(\epsilon^{-1})$ and $C=\mathcal{O}(\epsilon^{-1})$, we can estimate the two steady states as follows:\begin{align}
    (u^s_1, v^s_1, p^s_1, y^s_1) =(\fO(1),0,0,\fO(1)),~(u^s_2, v^s_2, p^s_2, y^s_2) =(\fO(\epsilon),\fO(1),\fO(\epsilon^{-1}),\fO(\epsilon)).
\end{align}
The concentration of $Y$ is $\fO(\epsilon^{-1})$ in the boundary steady state and is $\fO(1)$ in the interior steady state. Since the initial condition of $Y$ is $\fO(\epsilon^{-1})$, it can be expected that a significant time is needed if the system would like to reach the interior steady state for small $\epsilon$.

\begin{rmk}
    Note that when \(\epsilon \to 0\), we have $$\lim_{\epsilon \to 0} \vu_1 = (1,0,0,f), \quad \lim_{\epsilon \to 0} \vu_2 = \left(0, \frac{f}{k+f}, \infty, 0\right)\ .$$
  Hence, the boundary steady state in the variational Gray-Scott model corresponds to the uniform steady state of the classical Gray-Scott model, while the interior steady state is not related to any steady state of the classical Gray-Scott model.
\end{rmk}

In this section, we analyze the energy stability of two uniform steady states with fixed $\epsilon$. Although the variational Gray-Scott model comprises four species - $U$, $V$, $P$, and $Y$, the conservation law constrains the system's degrees of freedom to three, corresponding to three reaction trajectories $R_i$. Specifically, based on Eqs.~(\ref{law}), the perturbation in the stability analysis satisfies
\[\vs=(- \alpha_1 - \alpha_3, \alpha_1 - \alpha_2, \alpha_2, \epsilon \alpha_3)\ ,\]where $\alpha_i\in \sR$ for $i=1,2,3$. Thus we define the perturbation manifold as follows: 
\begin{equation}
    \fS_1:=\operatorname{span}\{\vs_1,\vs_2,\vs_3\},~\vs_1:=(-1,1,0,0),~\vs_2:=(0,-1,1,0),~\vs_3:=(-1,0,0,\epsilon).
\end{equation}
where, $\vs_i$ correspond to the perturbation along the reaction trajectory $R_i$. We use $E$ denote the free energy (\ref{free_energy_GS}) without spatial integration throughout this section.

\subsection{Stability of the interior steady-state}
For the interior steady state $\vu_2$, we can prove that it is a local minimizer of the free energy (\ref{free_energy_GS}) based on the following proposition:
\begin{prop}\label{statwo}
For any $\epsilon>0$, the interior steady state
\begin{equation}
\vu_2=(u^s_2, v^s_2, p^s_2, y^s_2) = \left(\frac{\epsilon}{\lambda} C, ~ \frac{1}{\lambda} C, ~ \frac{k+f}{\epsilon\lambda} C, ~ \frac{f\epsilon}{\lambda} C\right)
\end{equation}
is a local minimizer of the energy $E$ on the manifold {$ {\vu}_0 + \fS_1$}, where ${\vu}_0 = (u_0, v_0, p_0, y_0/\epsilon)$ is the initial condition.
\end{prop}

\begin{proof}
   Without loss of generality, we perturb $\vu_2$  along the direction of  $\vs_2$  of 
    the three-dimensional manifold $\fS_1$    
   (a similar analysis applies to other directions)  and  obtain:
\begin{equation}
\frac{\dd E[\vu_2+\delta \vs_2]}{\dd \delta} = \ln (p^{s}_2+\delta) + \sigma_p - \ln (v^{s}_2-\delta) - \sigma_v.
\end{equation}
By the definition of $\sigma_p$ and $\sigma_v$, we conclude that $\delta=0$ is the unique solution of  $\frac{\dd E[\vu_2+\delta \vs_1]}{\dd \delta}=0$, indicating the critical point. Furthermore, by computing the second variation, we have:
\begin{equation}
\frac{\dd^2 E[\vu_2+\delta \vs_2]}{\dd \delta^2}=\frac{1}{p^{s}_2+\delta}+\frac{1}{v^{s}_2-\delta}.
\end{equation}
Since $v^{s}_2,p^{s}_2>0$ for any $\epsilon>0$, $\frac{\dd^2 E[\vu_2+\delta \vs_2]}{\dd \delta^2}\big|_{\delta=0}>0$, indicating that this steady state is a local minimizer in the $\vs_1$ direction, similarly for other directions.
\end{proof}

Next, we examine the energy landscape near the interior steady state and  its evolution as $\epsilon$ approaches zero. We plot the energy landscape around the steady state in three directions on $\fS_1$ shown in Fig. \ref{sta} for $\epsilon = 0.01$ and $\epsilon = 0.0001$ respectively. While the behavior of $\vs_1$ remains stable, for $\vs_2$ and $\vs_3$, we observe a notable change as $\epsilon$ diminishes: the stable region contracts significantly. Essentially, the inflection point approaches the steady state in this direction. 
\begin{figure}[!h]
\includegraphics[width=0.48 \textwidth]{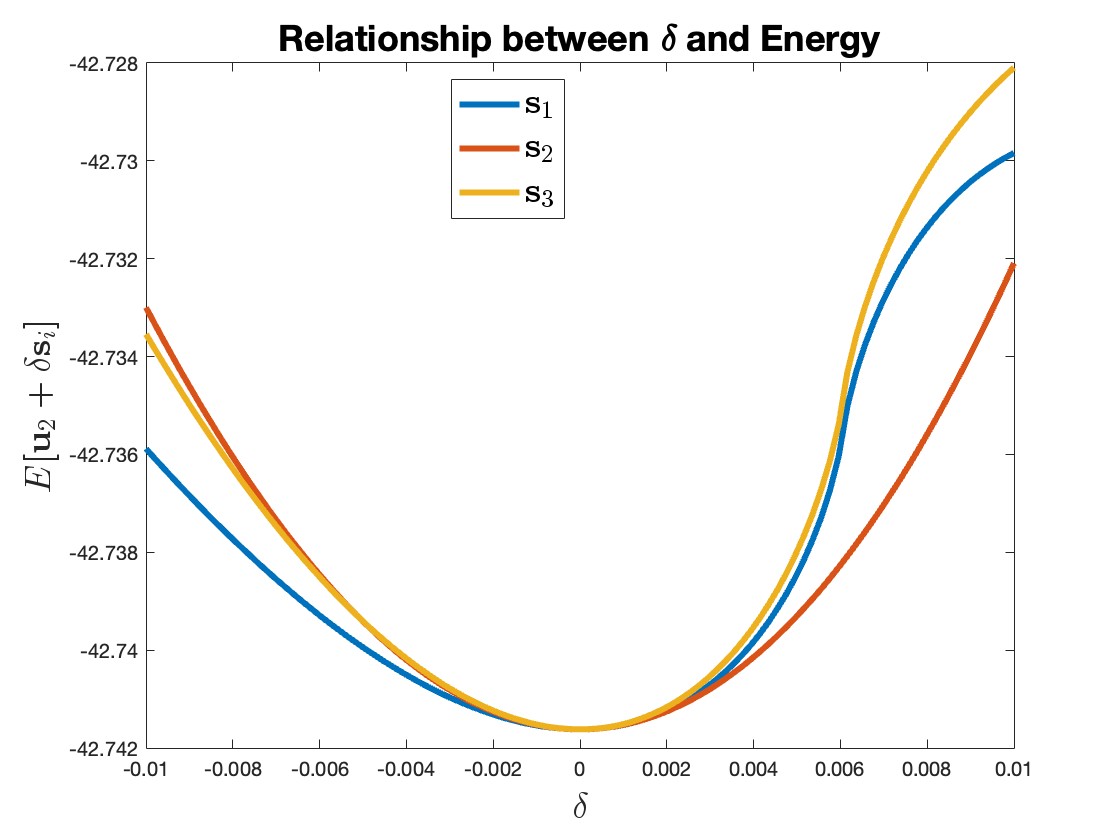}
\includegraphics[width=0.48 \textwidth]{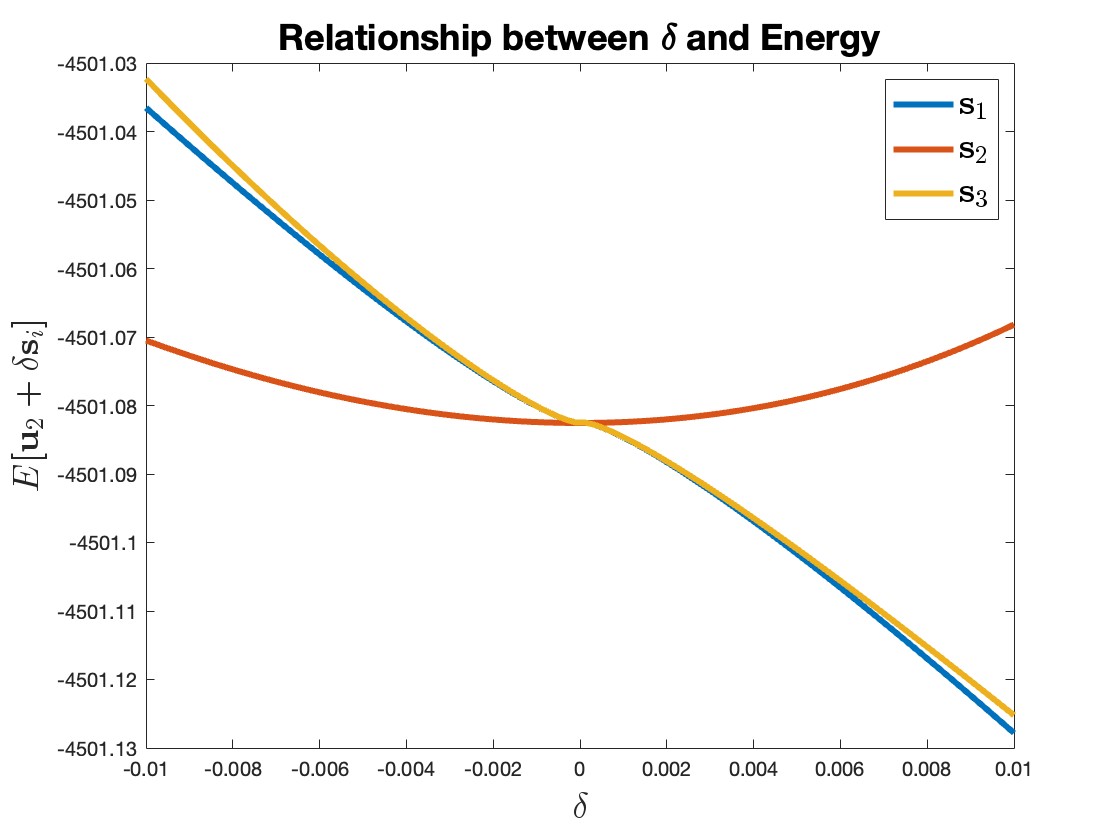}
\caption{Energy Landscape for $\epsilon = 0.01$ and $\epsilon_3=0.0001$ along the $\vs_i$ directions relative to the interior steady state.}\label{sta}
\end{figure}

The reason for this phenomenon is that $u_2^s$ will be close to $0$  when $\epsilon$ is small, as $u_2^s$ is at the order of $\mathcal{O}(\epsilon)$. Computing the second variation in the $\vs_1$ direction, we have
\begin{equation}
\frac{\dd^2 E[\vu_2+\delta \vs_1]}{\dd \delta^2}=\frac{1}{v^{s}_2+\delta}+\frac{1}{u^{s}_2-\delta}\ .
\end{equation}
Notice that $\frac{\dd^2 E[\vu_2+\delta \vs_1]}{\dd \delta^2} > 0$ requires $\delta < u_2^s$. Consequently, as $\epsilon$ decreases, the stability region around the local minimizer shrinks.


In summary, we have shown that the interior steady state remains stable for all $\epsilon>0$. However, as $\epsilon$ decreases, the stable region also becomes smaller. Nonetheless, it is still challenging for the system to escape from the steady state, as it would require $\delta > u^{s}_2$, which would result in $u<0$, an impossible scenario. Therefore, even for very small $\epsilon$, if the initial data is not far away from the steady state, the system will still converge to this interior steady state. However, when \(\epsilon\) is small, ensuring that the initial conditions are not too far from the steady state requires \(p\) to be of order \(\epsilon^{-1}\), meaning species \(P\) must have a significant presence. 

\subsection{Stability of the boundary steady-state}

For the boundary steady state, the stability analysis differs from that of the interior steady state. First, we observe that this is not a critical point of the free energy. Recall $\vs_1=(-1,1,0,0)\in \mathcal{S}_1$, then:
\begin{equation}
\frac{\dd E[\vu_1+\delta \vs_1]}{\dd \delta}=\ln (v^{s}_2+\delta) + \sigma_v - \ln (u^{s}_2-\delta) - \sigma_u = \ln \delta + \sigma_v - \ln (u^{s}_2-\delta) - \sigma_u.
\end{equation}
This expression tends to infinity as $\delta\to 0^+$, indicating it is not a critical point of the free energy. Indeed, we can compute the local minimizer of $E$ is the three-dimensional perturbation space $\vu_1 + \fS_1$, which satisfies:
\begin{equation}
  \begin{cases}
      \frac{u_1^s-\alpha_1 - \alpha_3}{\alpha_1-\alpha_2} = \epsilon \\ 
      \frac{\alpha_1-\alpha_2}{\alpha_2} = \frac{\epsilon}{f+k} \\
      \frac{\epsilon(u_1^s-\alpha_1-\alpha_3)}{y_1^s+\epsilon \alpha_3} = \frac{\epsilon}{f}.
  \end{cases}
\end{equation}
This system has a unique nonzero solution such that 
\[\vu_1 + \alpha_1 (-1,1,0,0) + \alpha_2(0,-1,1,0) + \alpha_3(-1,0,0,\epsilon)=\vu_2. \]
Hence, $\vu_2$ is the unique minimizer of the free energy $E$ in the space $\vu_1 + \fS_1$. The reason that ${\bm u}_1$ is a steady state of the system is that the mobilities for $R_1$ and $R_2$, given by $\epsilon v^3$ and $\epsilon p$, vanish, due to the disappearance of $V$ and $P$.
Consequently, the first and second reactions cease to exist:
\begin{equation}
\ce{U + 2 V <=>[1][\epsilon] 3V}, \quad \ce{V <=>[k$+$f][\epsilon] P}.
\end{equation}
We can show that the boundary steady state is a local minimizer of the free energy $E$ in a one-dimensional perturbation space \(\fS_2:=\operatorname{span}\{(-1,0,0,\epsilon)\}\), i.e., a perturbation along the direction of the third reaction trajectory:
\begin{prop}\label{virtual stability}
For any $\epsilon>0$, the boundary steady state
\begin{equation}
\vu_1=(u^s_1, v^s_1, p^s_1, y^s_1) := \left(\frac{\epsilon}{\epsilon+f} C,~ 0,~ 0, ~\frac{f\epsilon}{(\epsilon+f)} C\right),
\end{equation}
is a local minimizer of the free energy $\mathcal{F}$ on the manifold {$ {\vu}_0 + \fS_2$}, where ${\bf u}_0 = (u_0, v_0, p_0, y_0/\epsilon)$ is the initial condition.
\end{prop}
\begin{proof}
   The proof follows a similar structure to that of Proposition \ref{statwo}.
\end{proof}

Although the boundary steady state is only stable along the direction $(-1,0,0,\epsilon)$ and not stable in the $(-1,1,0,0)$ and $(0,1,-1,0)$ directions. It still behaves as a nearly local minimizer when $\epsilon$ is small and $\mathcal{O}(1)$ amounts of U, V, and P, and $\mathcal{O}(\epsilon^{-1})$ of Y present initially.
More precisely, since the first and second reactions account for only a fraction of $\epsilon$ in the system, 
it is reasonable to consider the perturbation along the direction
\[\vs =\vs_3 + \epsilon \vs_1 + \epsilon \vs_2 = (-1,0,0,\epsilon)+(-\epsilon,0,\epsilon,0),\] direct calculation reveals that:
\begin{equation}
\lim_{\epsilon \to 0} \frac{\dd E[\vu_1+\delta \vs]}{\dd \delta}=\lim_{\epsilon \to 0} \frac{\dd E[\vu_1+\delta \vs_3]}{\dd \delta}+\lim_{\epsilon \to 0} \frac{\dd E[\vu_1+\epsilon\delta (\vs_1+\vs_2)]}{\dd \delta}=\lim_{\epsilon \to 0} \frac{\dd E[\vu_1+\delta {(-1,0,0,\epsilon)}]}{\dd \delta}=0.
\end{equation}
Additionally, numerical results further support this assertion, as shown in Figure \ref{near}. The boundary steady state is a local minimizer in the one-dimensional space for small $\epsilon$. This is primarily because most reactions in this scenario are governed by the third reaction, which is stable for the boundary steady state. %
Consequently, for certain initial conditions, the system will converge to the boundary steady state when $\epsilon$ is small.  We term it as \textit{virtual stability}. The reason for labeling this as ``virtual'' is that the third reaction is artificial and virtual, as it does not exist in the classical Gray-Scott models defined by Eq.~(\ref{classical}). Moreover, when $\epsilon$ is not very small, it can be noticed that the boundary steady state is not a local minimizer, even within the restricted one-dimensional setting. As a result, a small perturbation around it leads to convergence toward the interior steady state. 
\begin{figure}[!h]
\includegraphics[width=0.48 \textwidth]{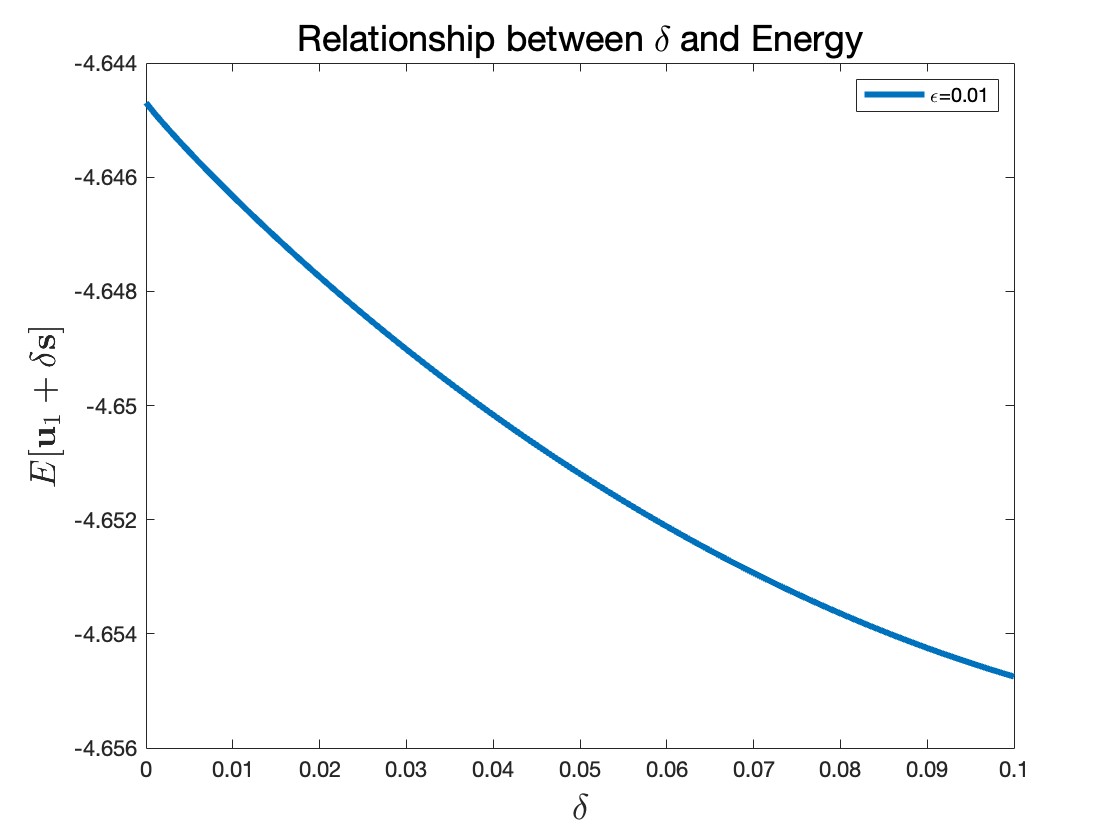}
\hfill
\includegraphics[width=0.48 \textwidth]{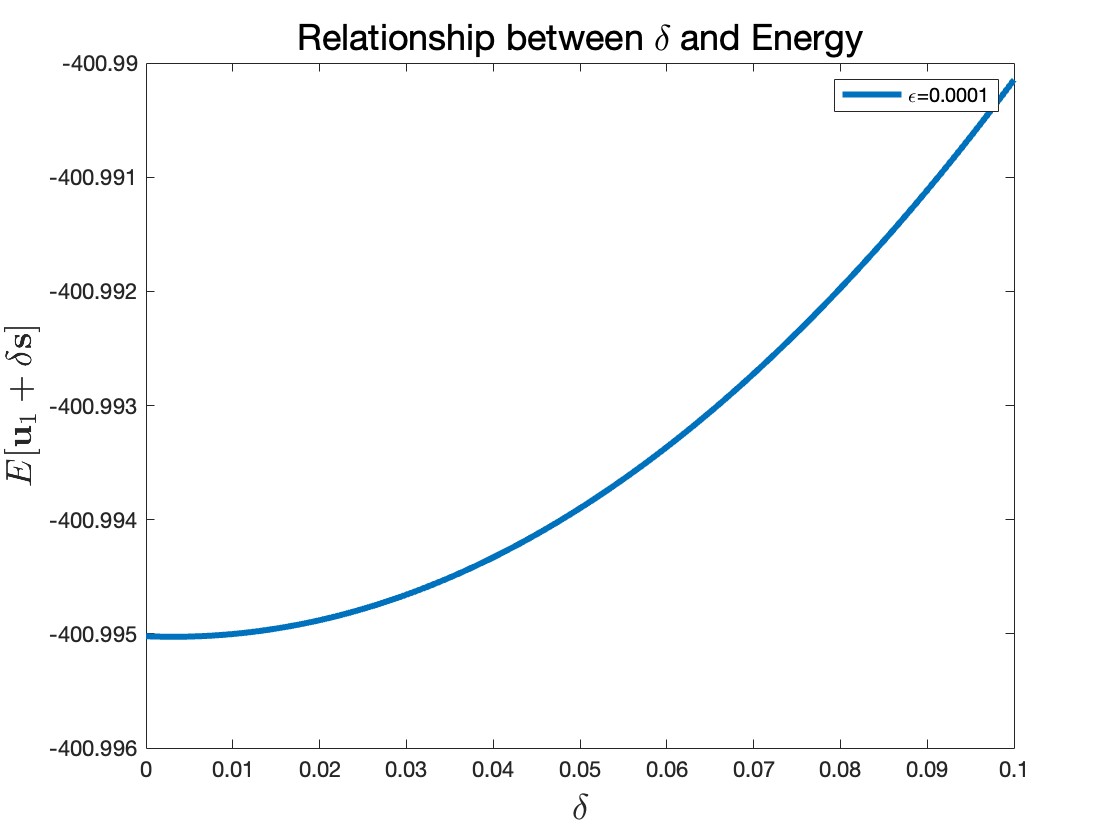}
\caption{Numerical results demonstrating near-stability when Y dominates almost the entire system.}\label{near}
\end{figure}

Overall, we observe that the interior steady state is stable, while the boundary steady state is considered virtually stable. The virtual stability arises when V and P disappear quickly in the system. In such cases, the initial terms of P and V occupy only a fraction of $\epsilon$ in the entire system, while the term Y dominates almost entirely. This scenario satisfies the conditions for virtual stability, leading the system to converge to the boundary steady state.



\section{Pattern formulation in Variational Gray-Scott Models}\label{experiment}

In this section, we explore the pattern formulation in the variational Gray-Scott model for different $\epsilon$ in one dimension.

For the classical Gray-Scott model (where $\epsilon = 0$), stationary spatial patterns correspond to non-uniform steady states of the system \cite{hao2020spatial}. Additionally, traveling patterns exist, which correspond to traveling wave solutions \cite{qi2017travelling}.
In \cite{hao2020spatial},  the authors compute steady states of the classical Gray-Scott model in the domain $\Omega = (0, 1)$ subject to the non-flux boundary condition under the standard finite difference discretization. The parameter values used are $D_u = 5 \times 10^{-4}$ and $D_v = 2.5 \times 10^{-4}$, $f = y_0 = 0.04$, and $k = 0.065$. 
It shows that, under these parameter values and the non-flux boundary condition, in addition to the trivial uniform steady state $(u, v) = (0, 1)$, the classical Gray-Scott model has $6$ non-uniform linearly stable steady states and $16$ linearly unstable nonuniform steady states, shown in Fig. \ref{Sol_in_GS}.
These steady states are computed using homotopy continuation techniques. Their linear stability is determined by analyzing the eigenvalues of the Jacobian matrix at the steady states. Specifically, the spectrum is computed numerically by solving the characteristic equation for the Jacobian. If any eigenvalue has a
positive real part, the steady state is unstable, while if all eigenvalues have negative real parts, the steady state is stable.

\begin{figure}[!hb]
  \centering
      \includegraphics[width= 0.9 \linewidth]{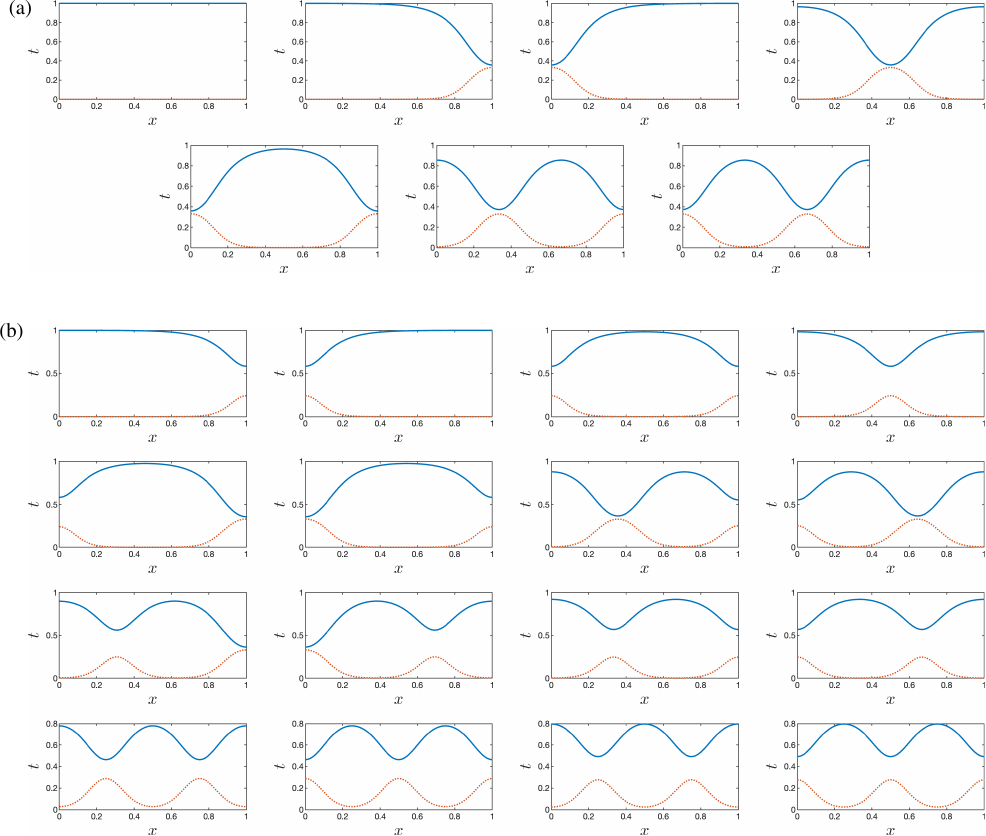}
      \caption{Steady states of the classical Gray-Scott model in the domain $\Omega = (0, 1)$, subject to the non-flux boundary condition, with parameter values $D_u = 5 \times 10^{-4}$ and $D_v = 2.5 \times 10^{-4}$, $f = y_0 = 0.04$, and $k = 0.065$, computed in \cite{hao2020spatial}: (a) Linearly stable steady states ($u$ : solid line, $v$: dashed line). (b) Linearly unstable steady states ($u$ : solid line, $v$: dashed line).}
      \label{Sol_in_GS}
  \end{figure}

\begin{figure}[!b]
  \centering
    \includegraphics[width = 0.9 \textwidth]{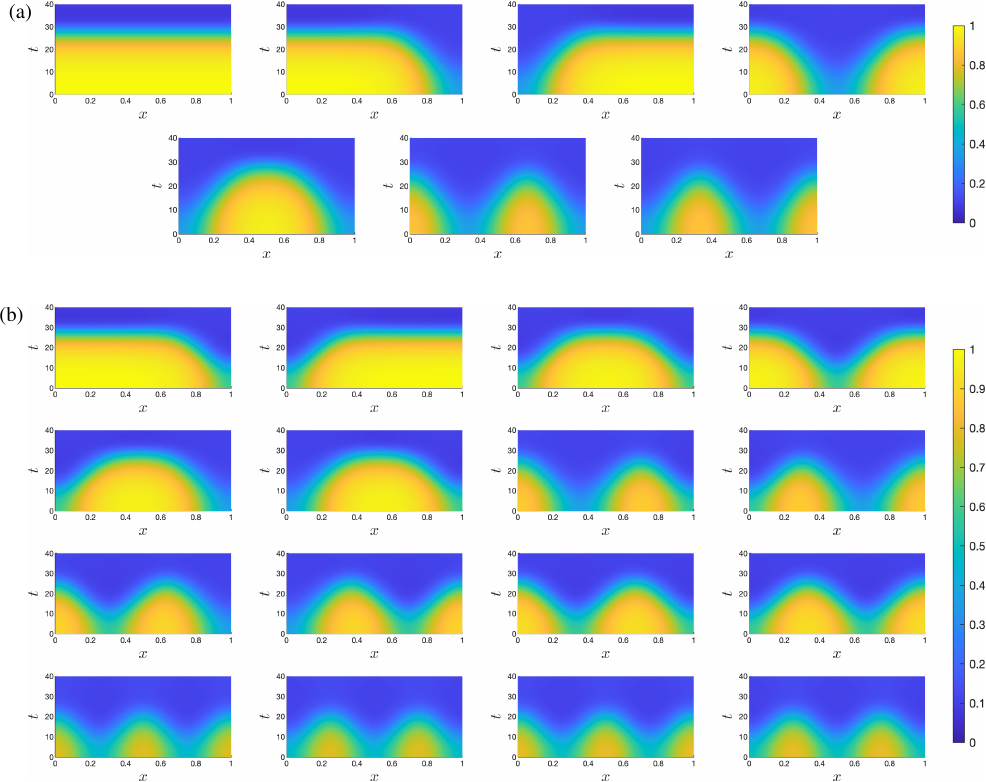}   
    \caption{(a) Numerical results for the variational Gray-Scott model with $\epsilon = 10^{-2}$ using linearly stable steady states of the classical Gray-Scott model as initial conditions. (b) Numerical results for the variational Gray-Scott model with $\epsilon = 10^{-2}$ using linearly unstable steady states of the classical Gray-Scott model as initial conditions}
    \label{Res_1e-2}
\end{figure}

To explore pattern formation in the variational Gray-Scott model for different $\epsilon$, we use these 23 steady states of the classical Gray-Scott model as the initial conditions for $(u, v)$. The initial conditions of $p$ and $y$ are taken as
\begin{equation}
p_0(x) = 1, \quad y_0(x) = f.
\end{equation}
We impose the same no-flux boundary conditions as in \cite{hao2020spatial}. We then examine the evolution of these initial conditions for various values of $\epsilon$, with the expectation that they serve as suitable starting points for capturing pattern formation in the variational model. These initial conditions can also be interpreted as spatially heterogeneous perturbations of the two uniform steady states analyzed in the previous section.

We adopt a semi-implicit method, which treats the reaction part explicitly and the diffusion part implicitly, for the temporal discretization. The spatial grid size is $h = 1/256$.
Since the initial concentration of species $Y$ is $y_0 / \epsilon$, for smaller $\epsilon$, more $Y$ exist and the total mass $C$, defined in (\ref{Def_C}), is also larger. We consider $\epsilon 
 = 10^{-2}$, $10^{-4}$, and $10^{-6}$, corresponding to a relatively large $\epsilon$, an intermediate-sized $\epsilon$,and a significantly small $\epsilon$, respectively.

\subsection{$\epsilon = 10^{-2}$: relatively large $\epsilon$}

We first consider a relatively large $\epsilon$ by taking $\epsilon = 10^{-2}$. Fig. \ref{Res_1e-2}(a) - (b) shows the time evolution of different initial conditions for $\epsilon = 10^{-2}$, visualized through $u(x, t)$ for $x \in (0, 1)$ and $t \in (0, 40)$. The temporal step size is set to $\Delta t = 0.5$. Numerical results show that decreasing the temporal step size yields quantitatively similar solutions. The initial conditions in Fig. \ref{Res_1e-2}(a) correspond to linearly stable steady states of the classical Gray-Scott model, while those in Fig. \ref{Res_1e-2}(b) correspond to linearly unstable steady states.

\begin{figure}[!b]
   \centering
   \includegraphics[width = 0.9 \textwidth]{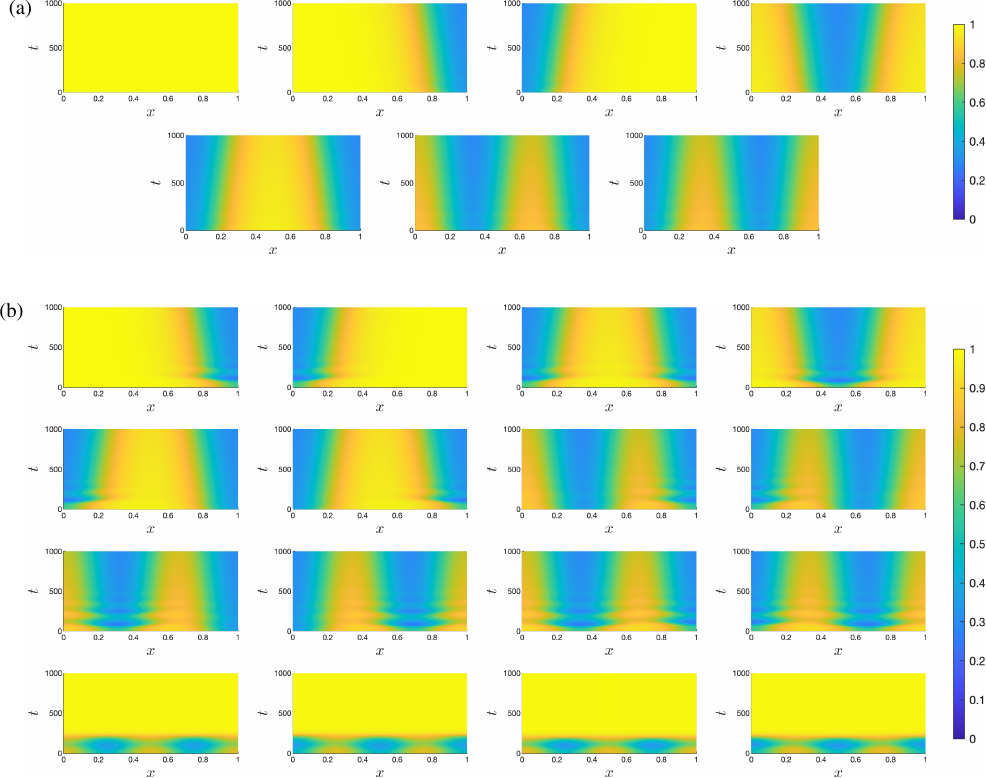}   
    \caption{(a) Numerical results for the variational Gray-Scott model with $\epsilon = 10^{-4}$ using linearly stable steady states of the classical Gray-Scott model as initial conditions. (b) Numerical results for the variational Gray-Scott model with $\epsilon = 10^{-4}$ using linearly unstable steady states of the classical Gray-Scott model as initial conditions.}
    \label{Res_1e-4}
\end{figure}

The simulation result shows that, for relatively large $\epsilon$, all initial conditions, including the one with $(u, v) = (1, 0)$, converge to the uniform interior steady state ${\bm u}_2$, as the concentration of $U$ will decrease to around $0$. As a result, any non-uniform patterns in the limiting system are destroyed in a relatively short time for a relatively large $\epsilon$. This result is consistent with the theoretical analysis presented in the previous section, as the interior steady state ${\bm u}_2$ is a minimizer of the free energy and the boundary steady state is not virtually stable for a relatively large $\epsilon$.

\subsection{$\epsilon=10^{-4}$: Intermediate-sized $\epsilon$}

Next, we consider an intermediate-sized $\epsilon$, $\epsilon = 10^{-4}$.  The time evolution of different initial conditions for $\epsilon = 10^{-4}$, visualized by $u(x, t)$ for $x \in (0, 1)$ and $t \in (0, 1000)$, are shown in Fig.\ref{Res_1e-4}(a) - (b). The temporal step size is set to $\Delta t = 0.5$. The initial conditions in Fig. \ref{Res_1e-4}(a) correspond to linearly stable steady states of the classical Gray-Scott model, while those in Fig. \ref{Res_1e-4}(b) correspond to linearly unstable steady states.

Unlike the case of $\epsilon = 10^{-2}$, the initial condition with $(u, v) = (1, 0)$ can converge to the boundary steady state. There are four initial conditions that converge to the boundary steady state ${\bm u}_1$ quickly. The result is consistent with the theoretical analysis that the boundary steady state is virtually stable for small $\epsilon$ .
Other initial conditions tend to converge to the interior steady state ${\bm u}_2$ as the concentration of $U$ decreases. However, the convergence is very slow, and non-uniform patterns can persist for a long time. Non-stationary patterns are still observable at $t = 1000$. For further illustration, Fig. \ref{fig:sol_u_1e-4}(a)-(b) show the profile of $u(x, t)$ at $t = 0, 250, 500, 750$ and $1000$ for two different initial conditions, corresponding to linearly stable solutions 4 and 6 in the classical Gray-Scott model. The results clearly show that $u(x, t)$ flattens over time and its total mass decreases in both cases.
\begin{figure}[!h]
    \centering
    \begin{overpic}[width=0.43\linewidth]{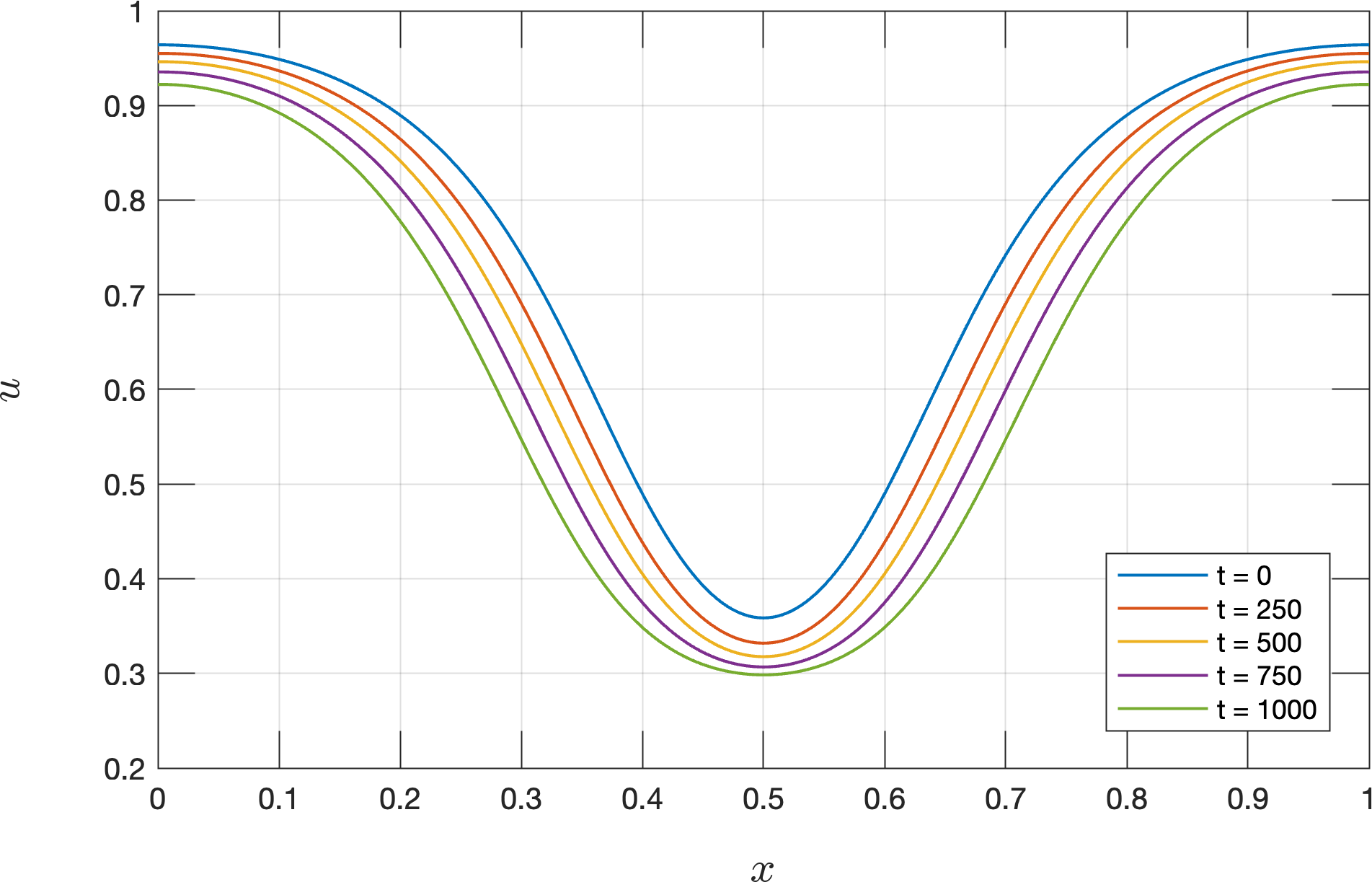}
    \put(-2, 60){(a)}
    \end{overpic}
    \hspace{2em}
    \begin{overpic}[width=0.43\linewidth]{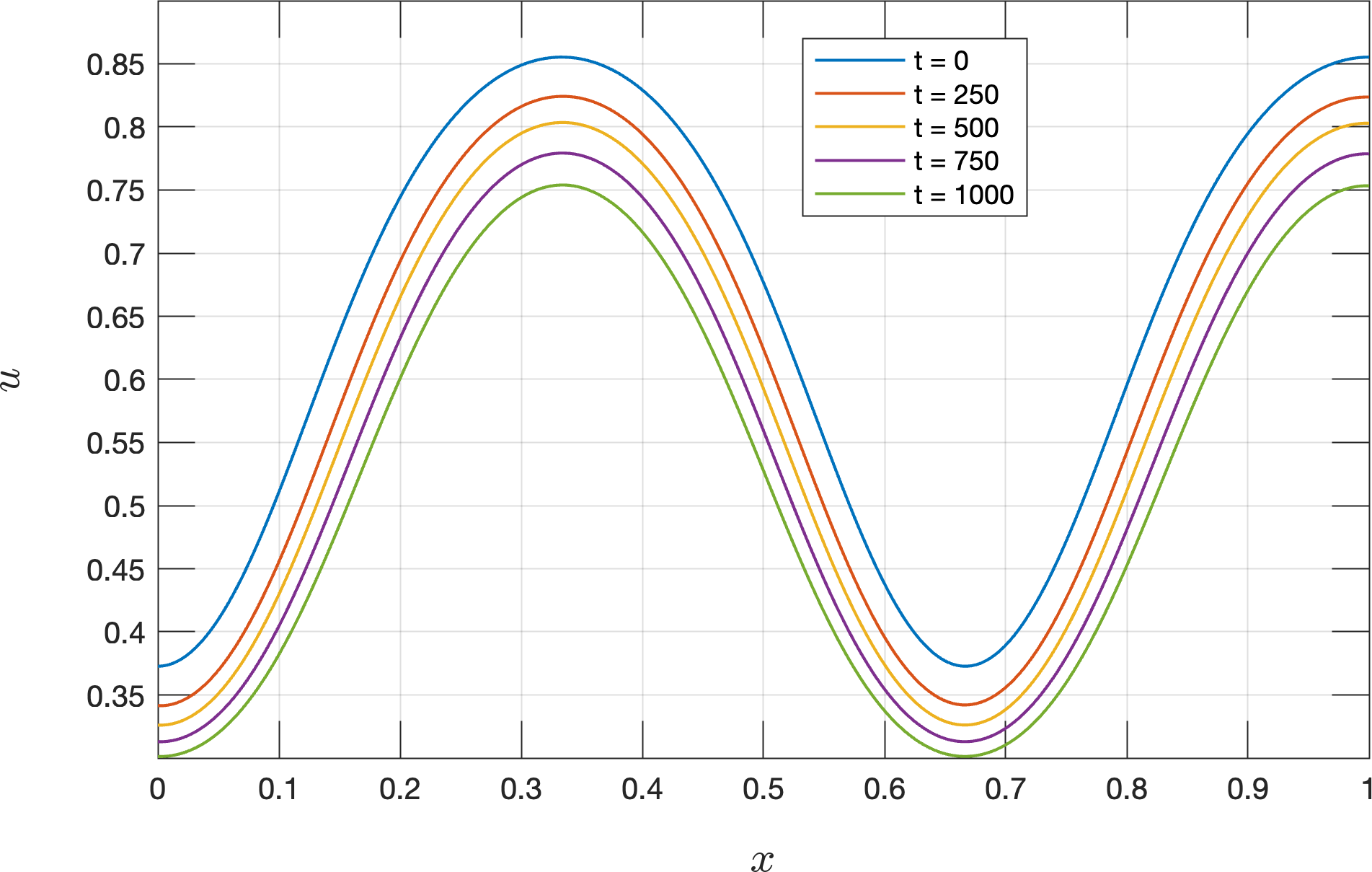}
      \put(-2, 60){(b)}
    \end{overpic}
    \caption{The profile of $u(x, t)$ at $t = 0, 250, 500, 750$ and $1000$ for two different initial conditions: (a) linearly stable solution 4 and (b) linearly stable solution 6.}
    \label{fig:sol_u_1e-4}
\end{figure}


More interestingly, for $\epsilon = 10^{-4}$, we observe oscillating solutions during the time evolution. In the original paper on the Gray-Scott model \cite{gray1984autocatalytic}, Gray and Scott demonstrate that the system exhibits chemical oscillations even without diffusion. 

\begin{figure}[!h]
   \centering
   \begin{overpic}[width = 0.4 \linewidth]{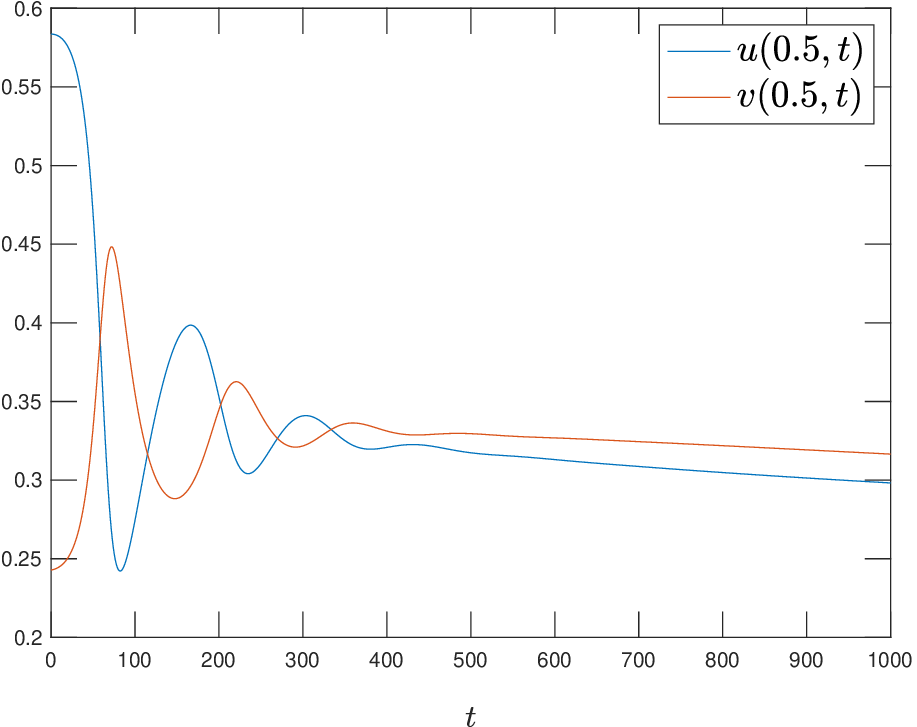}
     \put(-8, 75){(a)}
   \end{overpic}
   \hspace{3em}
   \begin{overpic}[width = 0.4 \linewidth]{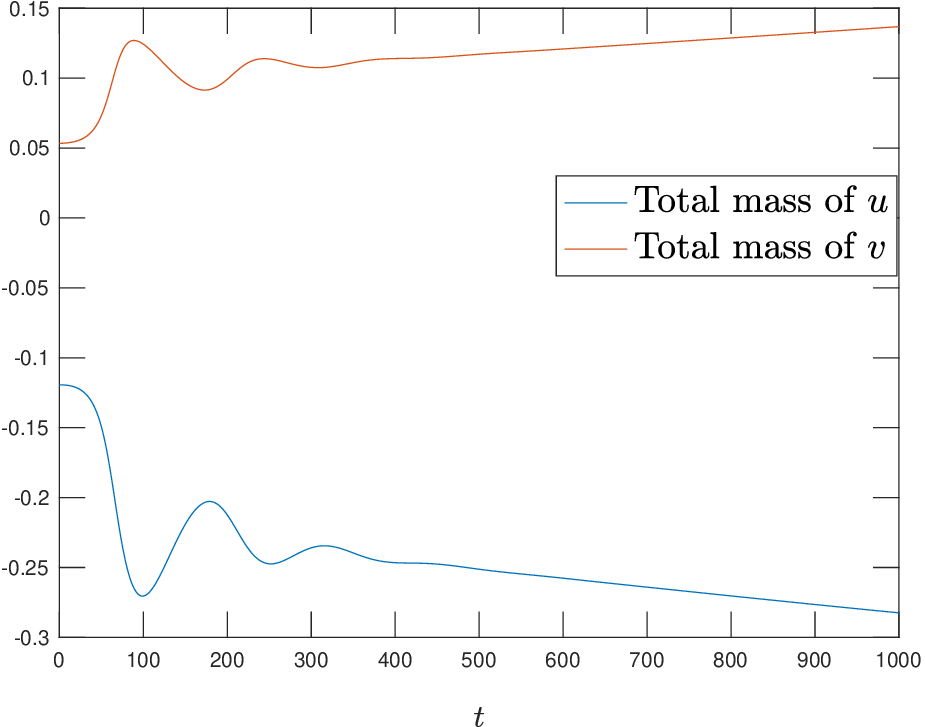}
       \put(-8, 75){(b)}
   \end{overpic}
      \caption{Oscillation behavior for the solution with initial condition corresponding to the unstable solution 4. (a) Evolution of $u(0.5, t)$ and $v(0.5, t)$ with respect to time. (b)  Evolution of total mass of $u$ and $v$ with respect to time. }\label{Res_os_1e-4}
\end{figure}

To illustrate the oscillation in the variational Gray-Scott model,  we examine the solution with the initial condition corresponding to the linearly unstable solution 4 (fourth image in Fig. \ref{Res_1e-4}(b)) in detail. Fig. \ref{Res_os_1e-4}(a)-(b) show the evolution of $u(0.5, t)$ and $v(0.5, t)$, as well as the total mass of $u$ and $v$, respectively. The plots show the damped oscillations in the concentration of $u$ and $v$, which are similar to the phenomenon reported in Fig. 4 in \cite{gray1984autocatalytic} for the irreversible Gray-Scott model without diffusion. After the oscillation, the solution behavior is similar to the solution with the initial condition corresponding to the stable solution 3 (fourth image in Fig. \ref{Res_1e-4}(a)).


\subsection{$\epsilon = 10^{-6}$: significantly small $\epsilon$}

Next, we consider $\epsilon = 10^{-6}$, a significantly small $\epsilon$. Fig. \ref{Res_1e-6} shows the simulation results for $\epsilon  = 10^{-6}$, visualized by $u(x, t)$ for $x \in (0, 1)$ and $t \in (0, 1000)$. Again, the initial conditions in Fig. \ref{Res_1e-6} (a) correspond to linearly stable steady states of the classical Gray-Scott model, while those in Fig. \ref{Res_1e-6} (b) correspond to linearly unstable steady states.

\begin{figure}[!h]
  \centering
   \includegraphics[width = 0.9 \textwidth]{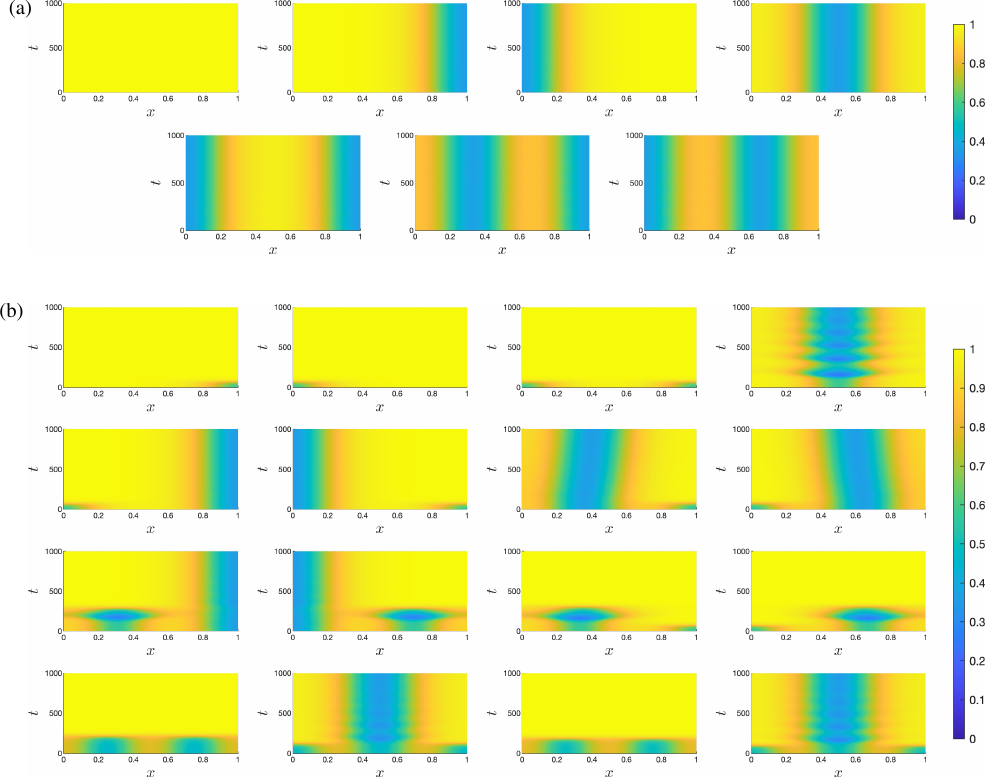}  
      \caption{(a) Numerical results for the variational Gray-Scott model with $\epsilon = 10^{-6}$ using linearly stable steady states of the classical Gray-Scott model as initial conditions. (b) Numerical results for the variational Gray-Scott model with $\epsilon = 10^{-6}$ using linearly unstable steady states of the classical Gray-Scott model as initial conditions.}\label{Res_1e-6}
      \end{figure}

Like the case with $\epsilon = 10^{-4}$, due to the virtual stability, the initial condition with $(u, v) = (1, 0)$  converges to the boundary steady state quickly. Additionally, more initial conditions will converge to the boundary steady state quickly, compared with the case of $\epsilon = 10^{-4}$. 

Interestingly, for the initial conditions corresponding to six non-uniform, linearly stable steady states in the classical Gray-Scott model, the profiles of $u$ and $v$ remain unchanged throughout the evolution. In other words, these linearly stable stationary patterns in the classical Gray-Scott model can be stabilized as {\bf transient states} in the variational model for a very long time when $\epsilon$ is significantly small. We can view these solutions as  {\bf quasi-steady states} or quasi-stable patterns as the $(u, v)$-component of the solution is unchanged. 
Fig. \ref{PandX} shows the concentrations of $U$, $V$, $P$, and $Y$ at $t = 0$ and $t = 1000$ for the initial condition associated with the non-uniform linearly stable solution 5. It can be observed that although the concentrations of $U$ and $V$ remain unchanged, the concentration of $P$ decreases while $Y$ increases by the same amount. The effective dynamics of the whole system is to transform $Y$ to $P$.
From the numerical experiments, 
one can expect the system will reach the interior steady \(u_2\) at the end as the concentration of $P$ increases and $Y$ decreases.
However, since the initial concentration of $Y$ is much larger than that of $P$, a significant time is needed to reach the steady state. Consequently, the spatial pattern can be stabilized for a long time if the concentration of $Y$ is large, i.e., $y(x) \approx f$. Additionally, there are four initial conditions, corresponding to linearly unstable steady-state 5, 6, 9, and 10, which will initially evolve towards a quasi-steady state and will remain unchanged in the $(u, v)-$ components for a long time.
The simulation results indicate that the variational Gray-Scott model can capture the formation of a stationary pattern in the irreversible Gray-Scott model.
\begin{figure}[!h]
\centering
\begin{overpic}[width = 0.49 \linewidth]{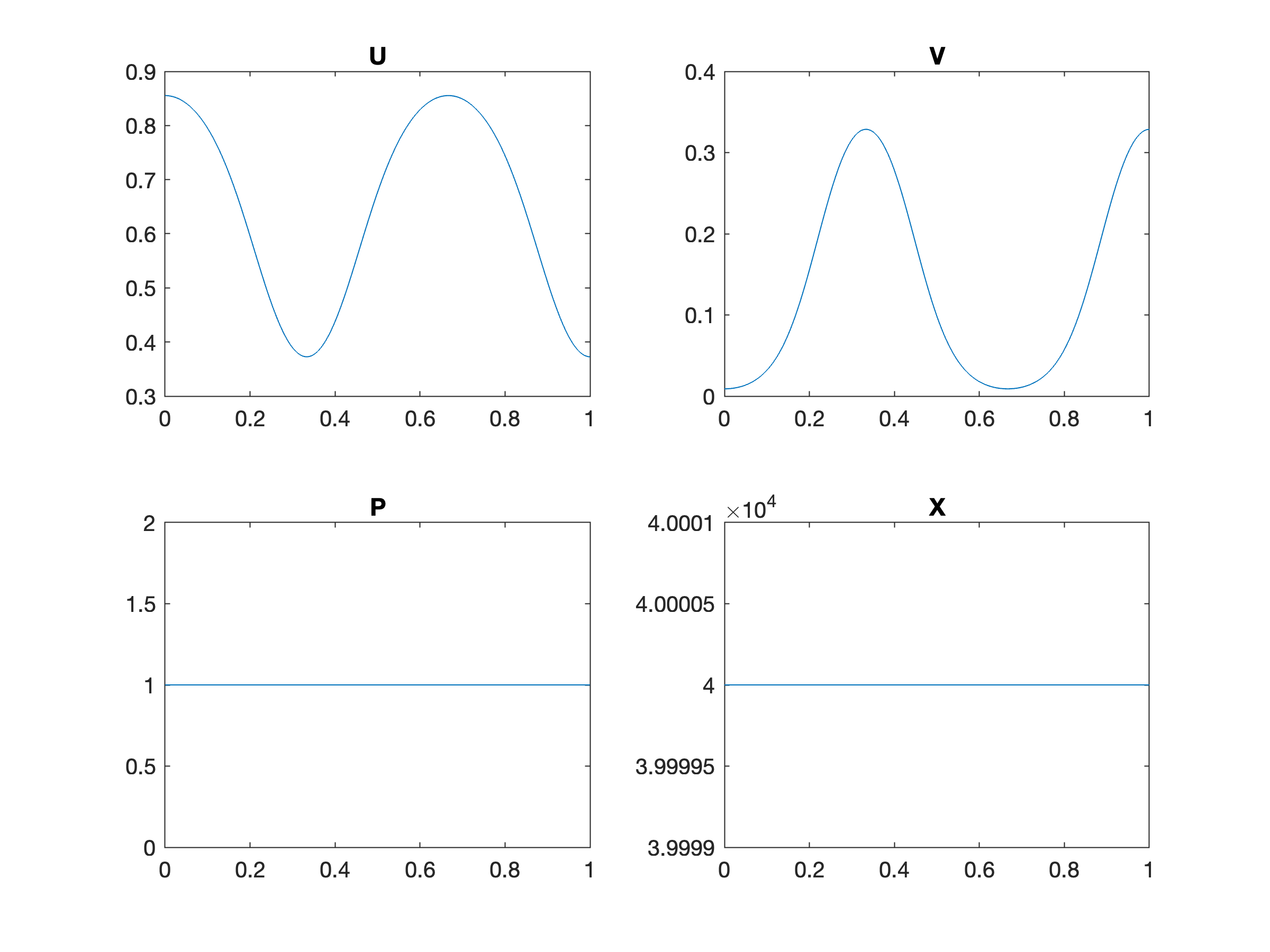}
\put(0, 65){(a)}
    \end{overpic}
\hfill
\begin{overpic}[width = 0.49 \linewidth]{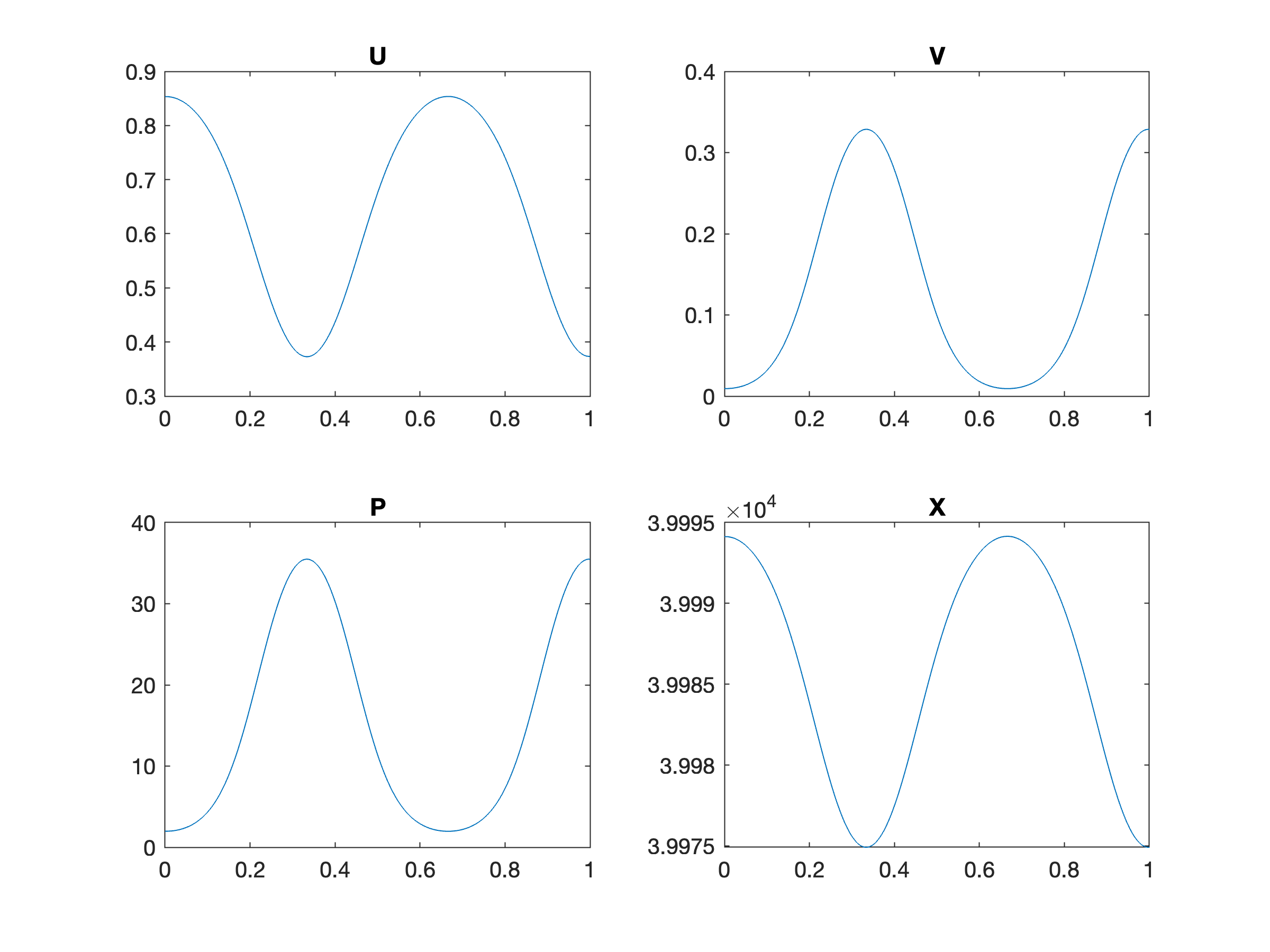}
               \put(0, 65){(b)}
    \end{overpic}
\caption{Concentration of each species at $t = 0$ and $t = 1000$. Clearly, the profile of $u$ and $v$ is almost unchanged. But the concentration of $P$ increases significantly and $Y$ decreases. }\label{PandX}
\end{figure}

\begin{figure}[!b]
   \centering
   \begin{overpic}[width = 0.4 \linewidth]{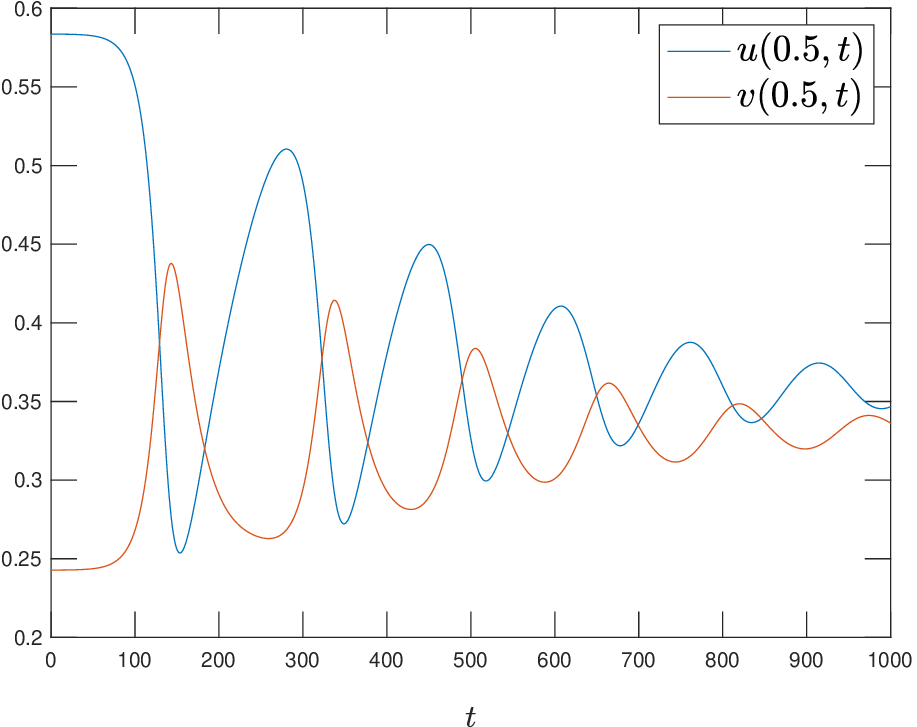}
      \put(-8, 75){(a)}
   \end{overpic}
   \hspace{3em}
   \begin{overpic}[width = 0.4 \linewidth]{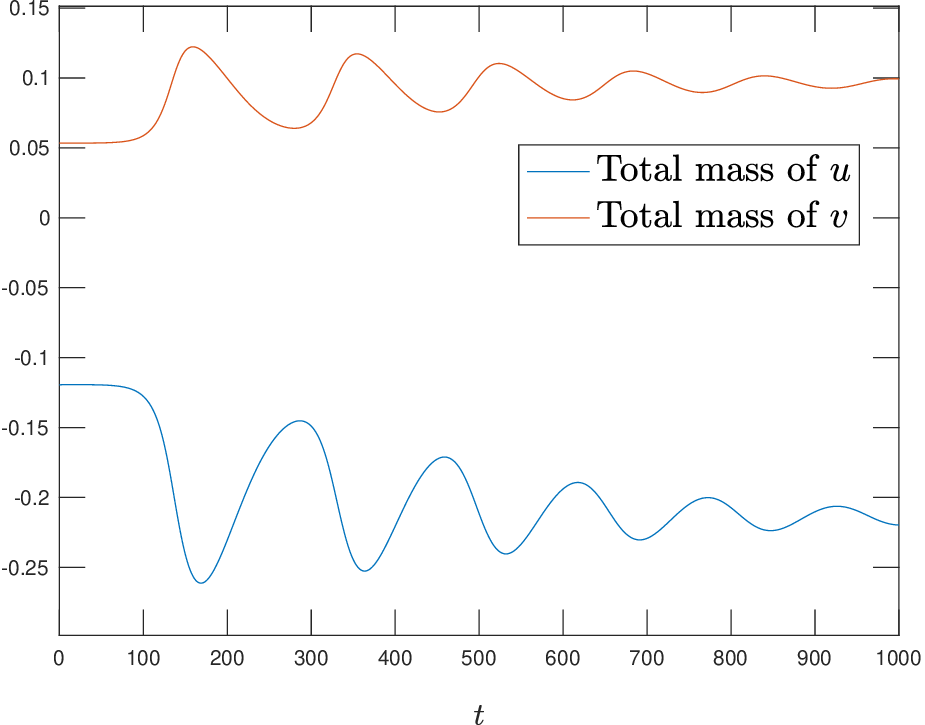}
      \put(-8, 75){(b)}
   \end{overpic}
      \caption{Oscillation behavior for the solution with initial condition corresponding to the unstable solution 4. (a) Evolution of $u(0.5, t)$ and $v(0.5, t)$ with respect to time. (b)  Evolution of total mass of $u$ and $v$ with respect to time.}\label{Res_1e-6_os}
\end{figure}

Similar to the case with $\epsilon = 10^{-4}$, we observe damped oscillation in the concentrations of $u$ and $v$ when the initial condition corresponds to unstable steady-state 4, 14, 16.  Fig. \ref{Res_os_1e-4}(a)-(b) show the evolution of $u(0.5, t)$ and $v(0.5, t)$, as well as the total mass of $u$ and $v$, respectively, for the solution with the initial condition corresponding to the unstable solution 4. Compared with $\epsilon = 1e-4$, the damping effect is smaller and more oscillations can be observed. All oscillating solutions will reach stable solution 3 after a long time.

Moreover, we observed two traveling-wave-like solutions when the initial condition corresponds to unstable solutions 7 and 8.
The existence of traveling wave solutions is one of the most interesting phenomena in the classical Gray-Scott model \cite{manukian2015travelling, doelman1997pattern, kyrychko2009control, qi2017travelling}. These are solutions with the form $u(x, t) = U(z), v(x, t) = V(z)$ with $z = x - ct$ for some constant $c$, which is known as wave speed. As shown in Fig. \ref{Res_1e-6}, after some initial evolution, two solutions exhibit a traveling-wave-like behavior. Fig. \ref{TW} shows the profile of $u$ in one of the traveling wave-like solutions at $t = 200. 400, 600, 800$ and $1000$. Strictly speaking, it is not a traveling wave, as the profile of $u$ is also changed slightly. This might be due to the boundary effect. Fig. \ref{TW}(b) shows the location of $\min (u)$, which shows a travelling-wave-like behavior. It remains an open question whether real traveling wave solutions exist in the variational Gray-Scott model, which will be investigated in future work.

\begin{figure}[!h]
\centering
     \begin{overpic}[width = 0.45\linewidth]{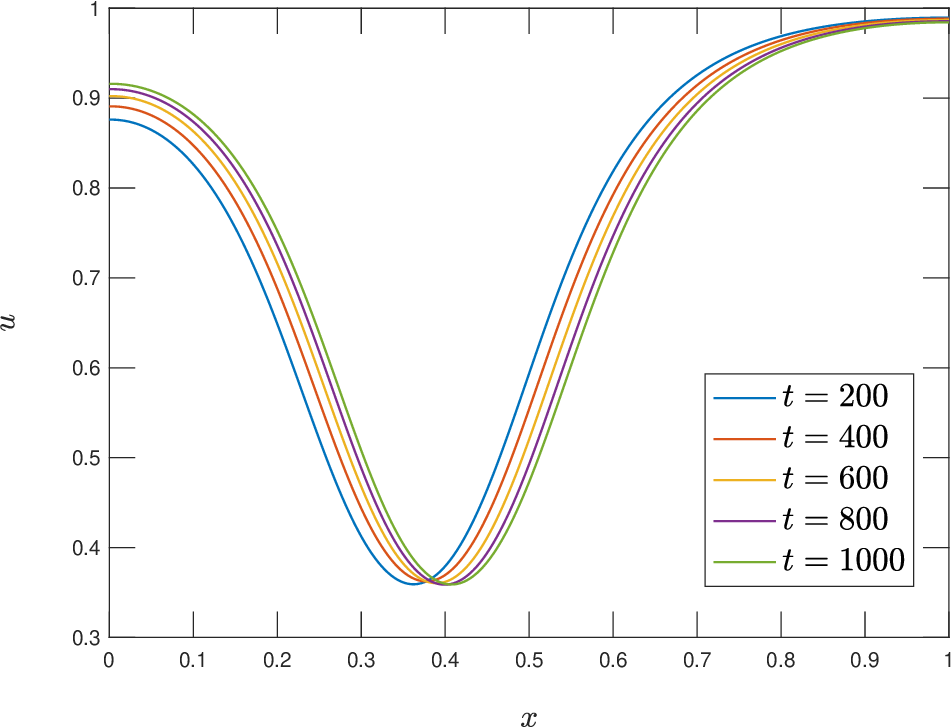}
     \put(-2, 75){(a)}
     \end{overpic}
   \hspace{3 em}
   \begin{overpic}[width = 0.43 \linewidth]{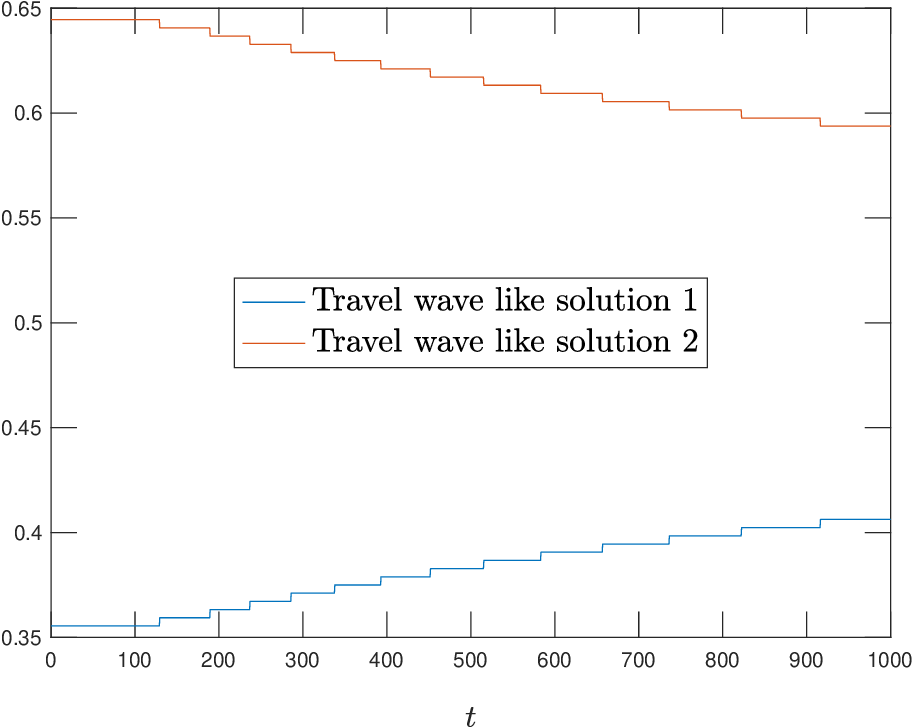}
     \put(-8, 76){(b)}
     \end{overpic}
      \caption{ (a). A traveling-wave-like solution for $\epsilon = 10^{-6}$, shown by profile of $u$ at $t = 200. 400, 600, 800$ and $1000$.  (b)Location of $\min (u)$ in two traveling-wave-like solutions. }\label{TW}
\end{figure}

The simulation result suggested that the variational Gray-Scott model can capture various pattern formation phenomena in the classical Gray-Scott model, including steady patterns, oscillating patterns, and traveling-wave-like patterns when $\epsilon$ is significantly small. These patterns appear in the variational model as transient states before the system reaches the uniform steady state.
It is worth mentioning that the classical two-species Gray-Scott model is a subsystem of the four-species variational Gray-Scott model. In this variational model, the concentration of species $Y$ is significantly larger than that of the other species, so the dynamics of the system are predominantly governed by the effective reaction $Y \ce{<=>} P$.
The initial condition of $U$ and $V$ will determine the dynamics either transform $Y$ to $P$ or $P$ to $Y$. In the current study, since the concentration of $P$ is taken as $1$, which is much smaller than that of $Y$, the system will converge to the boundary steady state quickly if the essential dynamics is $P$ to $Y$. We can observe the pattern formation for a significantly long time if the essential dynamics is $Y$ to $P$. We will investigate the effects of the initial concentration of $P$ in future work.

Furthermore, the simulation results are consistent with the energy stability analysis presented in previous sections. When $\epsilon$ is large, all initial conditions converge to the interior steady state $\vu_2$, which minimizes the free energy. However, when $\epsilon$ is sufficiently small, the gradient flow dynamics favor the boundary steady state $\vu_1$ for certain unstable initial configurations, due to the virtual stability of the boundary state in this regime. In this sense, both trivial steady states become effectively "stable." For some unstable steady states, the system tends to evolve toward the boundary steady state. Particularly, when the initial condition is close to the boundary steady state but far from the interior one—particularly due to the large value of $Y$—the system rapidly converges to the boundary state. In contrast, convergence to the interior steady state is much slower, leading to the persistence of non-uniform patterns.

\section{Pattern persistence time v.s. $\epsilon$ }
In this section, we study the pattern persistence time in the variational Gray-Scott model with respect to $\epsilon$. 

We consider two typical initial conditions, which correspond to the non-uniform linearly stable solution 6 (the last solution in Fig. \ref{Sol_in_GS}(a)) and the linearly unstable solution 3 (the third solution in Fig. \ref{Sol_in_GS}(b)). The first one will always converge to the interior steady state according to the numerical simulation. The second one will converge to the interior steady state for large $\epsilon$, but converge to the boundary steady state for small $\epsilon$. We define the pattern time as 
\[ T  = \inf \{ t ~|~ \max\{\max_{x} u(x, t)-\min_{x} u(x, t),~ \max_{x} v(x, t) - \min_{x} v(x, t)\} \leq a \}  \]
with $a$ being significantly small since the vanishing of the pattern indicates that both $u$ and $v$ are close to constant functions. We pick $a = 0.05$ throughout this section.
Fig.~\ref{ep} shows the relationship between the pattern persistence time and \(\epsilon\) for these two initial conditions.

\begin{figure}[!h]
\centering
\includegraphics[width=45em]{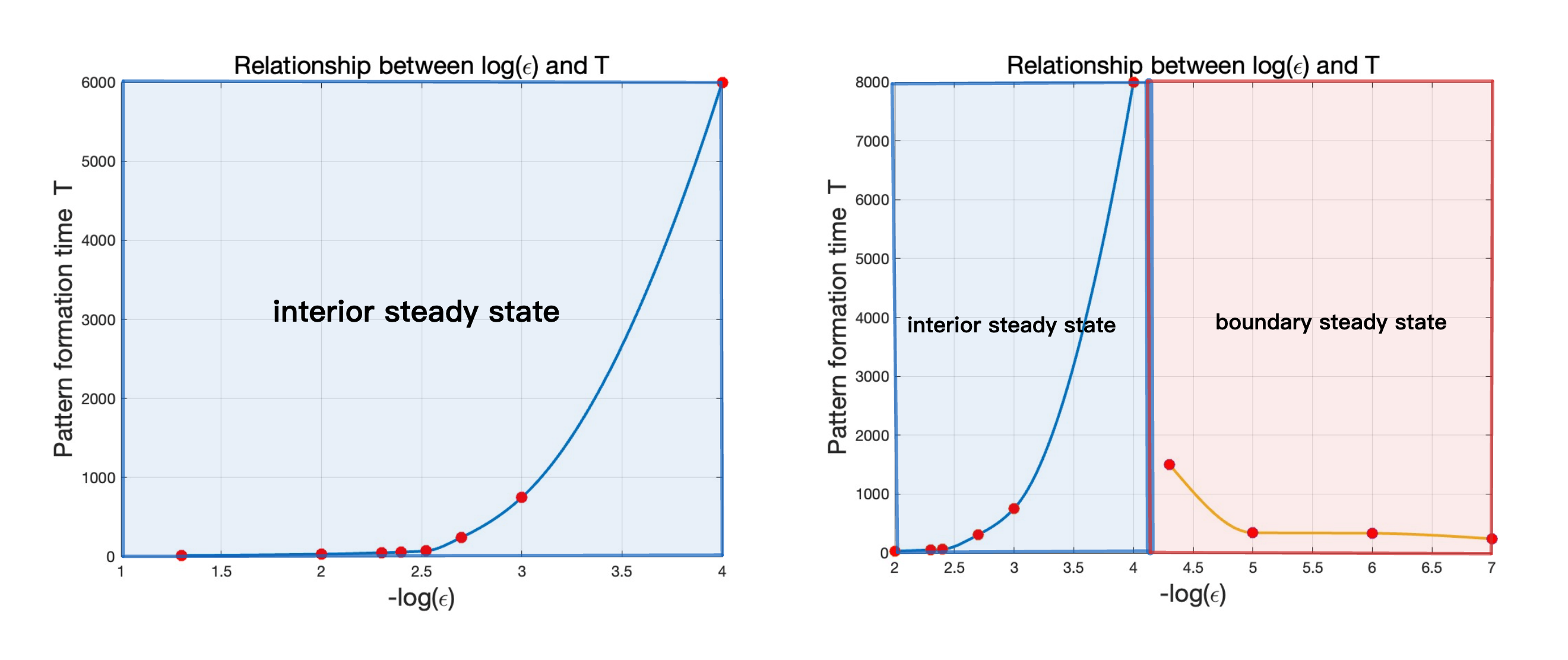}
\caption{The Pattern persistence time and \(\epsilon\) for a stable and an unstable initial condition, where pattern formation time is the first time when \(\max\{\max_x u(x, t) - \min_x u(x, t), \max_x v(x, t) - \min_x v(x, t)\} \leq 0.05\).}
\label{ep}
\end{figure}

For the linearly stable solution 6, the pattern persistence time increases as $\epsilon$ decreases, and is of the order $\fO(\epsilon^{-1})$. For the linearly unstable solution 3, it converges to the interior state for relatively large $\epsilon$, and the pattern persistence time also increases with decreasing $\epsilon$, following the order $\fO(\epsilon^{-1})$. However, the initial condition converges to the boundary steady state for large $\epsilon$, and the initial pattern will be destroyed faster for smaller $\epsilon$. The result is consistent with the previous simulation and analysis. 

We plot the evolution of free energy for two initial conditions with different values of $\epsilon$.
For $\epsilon = 10^{-2}$, the free energy plot clearly shows that both initial conditions converge to equilibrium around $t=400$. The evolution of free energy for $\epsilon = 10^{-4}$ follows a similar pattern, but it takes significantly longer for both solutions to reach the trivial steady state, as shown in Fig. (\ref{ep}). However, for $\epsilon = 10^{-6}$, the linearly unstable steady state in the classical Gray-Scott model quickly converges to the boundary steady state, although the free energy of this state remains relatively large. For the linearly stable state in the classical Gray-Scott model, while the $(u,v)$-components of the solution remain unchanged, as shown in Fig. \ref{Res_1e-6}, the system's free energy continuously decreases over time. The free energy plot also suggests that the linearly stable stationary patterns in the classical Gray-Scott model serve as transition states in the variational model.
\begin{figure}[!h]
    \centering
    \includegraphics[width=0.32\linewidth]{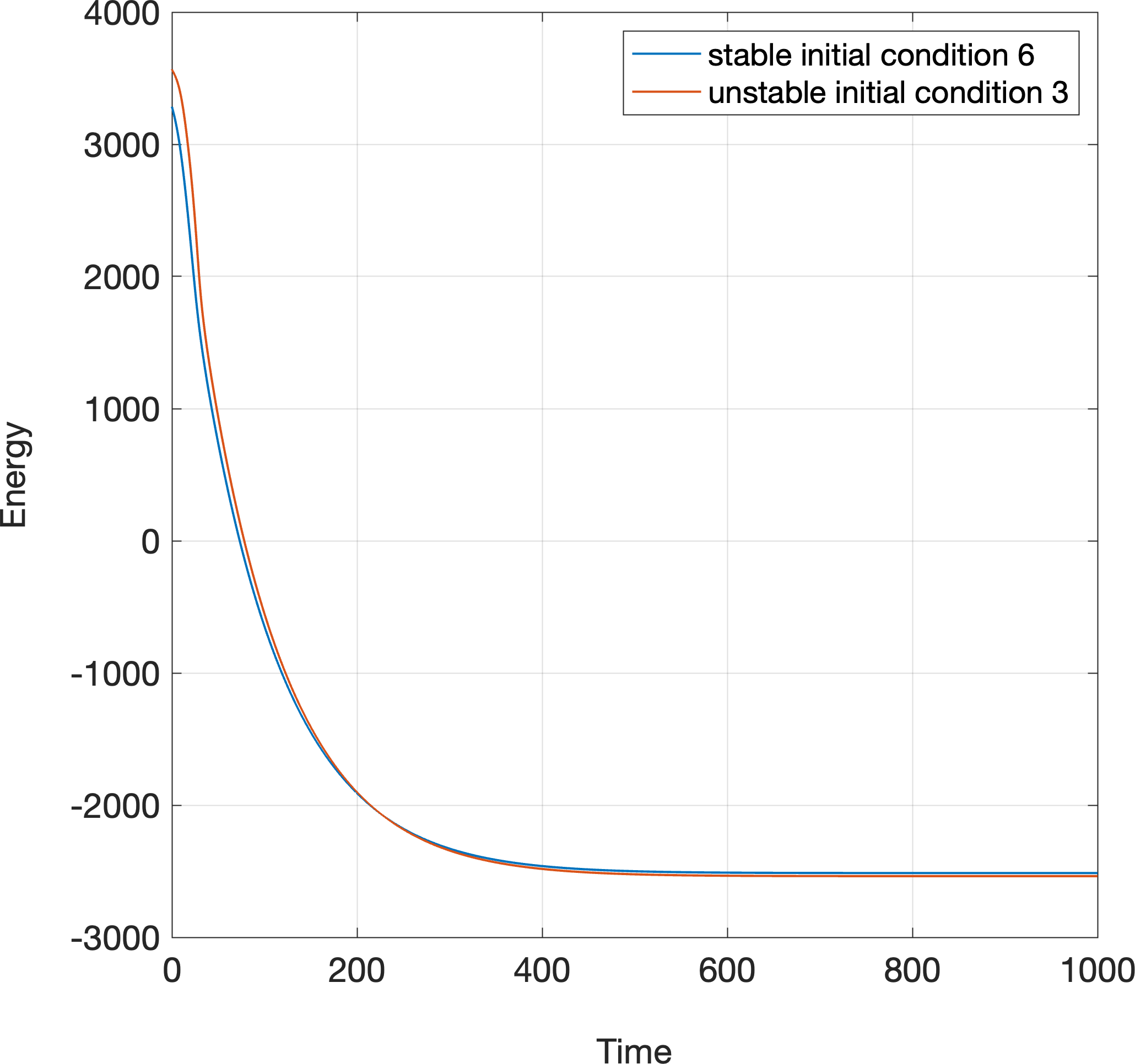}
    \hfill
    \includegraphics[width=0.32\linewidth]{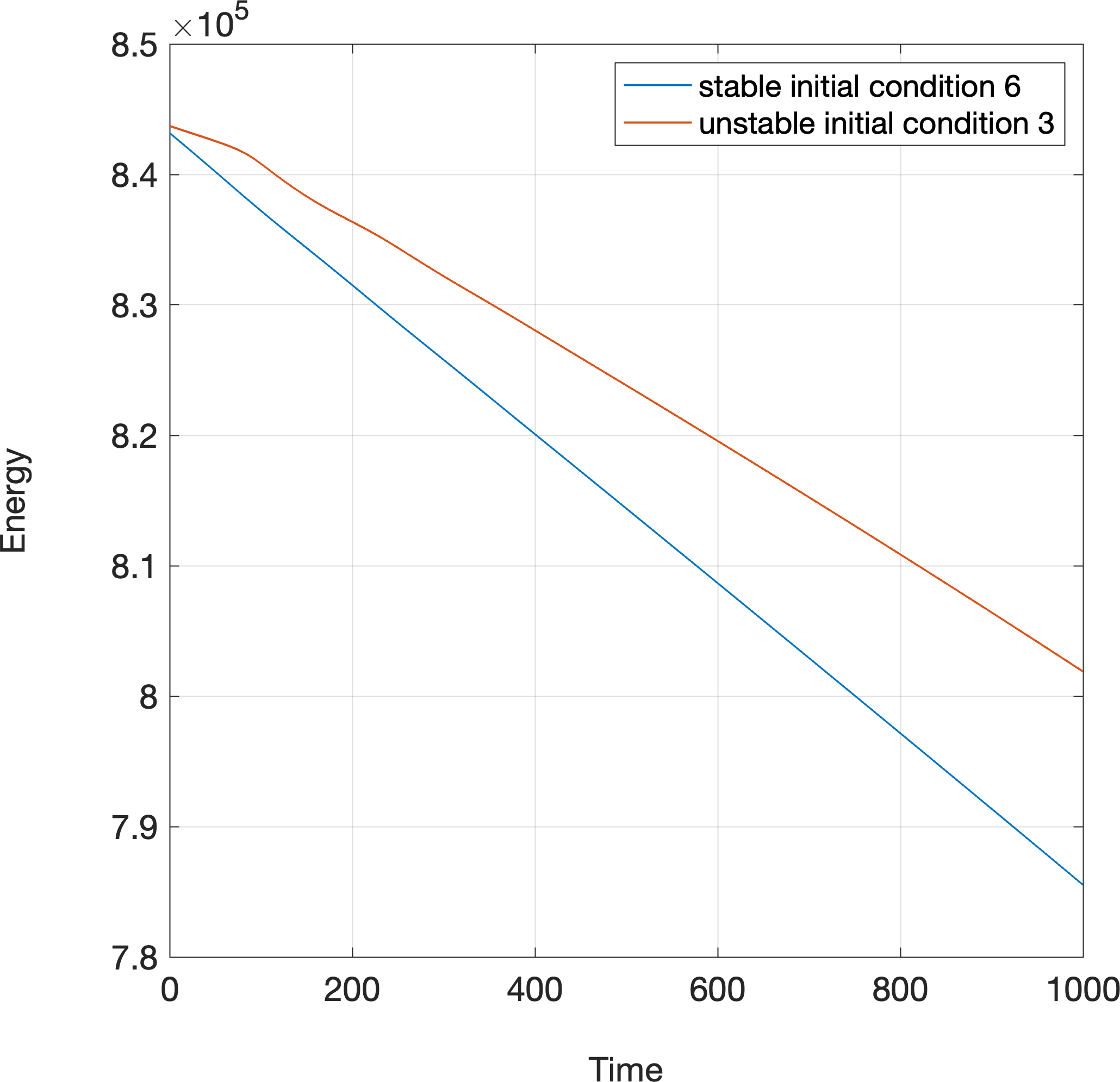}
    \hfill
    \includegraphics[width=0.32\linewidth]{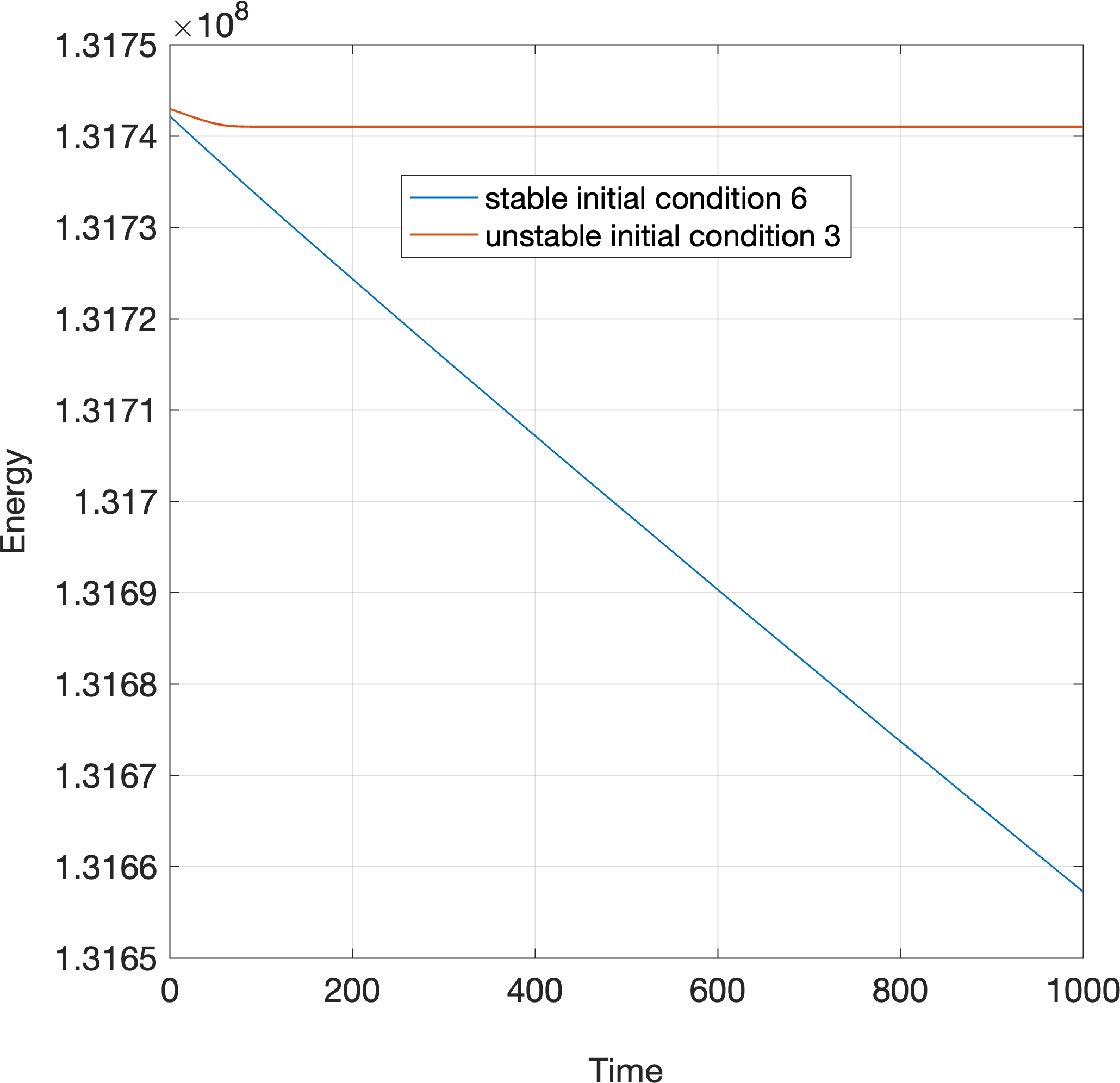}
    \caption{Evolution of free energy of two initial condition (stable 3 and unstable 6) for different value of $\epsilon$ (from left to right: $\epsilon = 10^{-2}, 10^{-4}, 10^{-6}$).}
    \label{fig:enter-label}
\end{figure}


To summarize, in order to have pattern formation in the variational Gray-Scott model, one needs to have some particular initial condition such that the variational model converges to the interior steady state, in which all $Y$ will convert to $P$. The pattern is maintained if the concentration of $Y$ stays large, analogous to the continuous feed of $U$ in the classical Gray-Scott model. The conclusion is similar to that in a recent paper \cite{zhang2023free} on a slightly different thermodynamically consistent three-species reaction-diffusion model, in which the authors show that for a finite system, a specific Turing pattern exists only within a finite range of total molecule number, and the presence of the third species stabilizes the Turing pattern of the two species.

\section{Conclusions}
In this paper, we study the pattern formation of a thermodynamically consistent variational Gray-Scott model, derived by an energetic variational approach, in one dimension, using non-uniform steady states of the classical model computed in \cite{hao2020spatial} as initial conditions. The variational Gray-Scott model includes a virtual term \( Y \) and reversible reactions to the classical Gray-Scott model, transforming the system into a thermodynamically consistent closed system. The classical Gray-Scott model can be viewed as a subsystem of the variational Gray-Scott model when the reverse reaction rate \(\epsilon\) tends to zero. By decreasing $\epsilon$, we observed that the stationary pattern in the classical Gray-Scott model can appear as the transient state in the variational model when $\epsilon$ is significantly small. Additionally, the variational model admits oscillating and traveling-wave-like solutions for small $\epsilon$. The results show the capability of the variational model in capturing pattern formation.

We also analyze the energy stability of two uniform steady states, an interior steady state and a boundary steady state, in the variational Gray-Scott model. Although the interior steady state is always stable, the stability region becomes significantly smaller as \(\epsilon\) decreases. In the meantime, the boundary steady state is virtually stable, i.e., is stable with respect to the third reaction $U \ce{<=>} Y$. For certain initial conditions, since the concentration of $Y$ (at order $\fO(\epsilon^{-1})$) is much larger than other species (at the order of $\fO(1)$) for small $\epsilon$, the gradient flow dynamics will drive the system to the boundary steady state. In order to observe pattern formation, one needs a special initial condition such that the dynamics will not converge to the boundary steady state.

The variational Gray-Scott model offers a new mathematical framework for understanding pattern formation from a thermodynamic perspective. The numerical simulation and theoretical analysis suggest that pattern formation is maintained by the presence of species $Y$, i.e., the continuous input of $U$ in the classical Gray-Scott model. The initial condition determines the effective direction of the reaction network, and the pattern formation can only occur if the network continuously generates the inert product $P$. Furthermore, for sufficiently small $\epsilon$, the system's free energy decreases over time due to the conversion of $Y$ to $P$, indicating that these patterns act as transition states in a larger, closed system. This also implies that maintaining the pattern requires a continuous energy input from the environment.

The current study represents a first step toward understanding pattern formation in biological systems from an energetic perspective. Several open questions remain for the variational Gray-Scott model, including: 1) Investigating the effect of varying the scale of small parameter $\epsilon_i$ in the reaction network (\ref{KCR_GS_Re}); 2) Analyzing the role of $P$ in pattern formation; 3) Examining more general initial conditions. We will study these open questions in future work.

\section*{Acknowledgement}

CL was partially supported by NSF DMS-2153029, DMS-2118181 and DMS-2410742.  YW was
supported by NSF DMS-2153029 and DMS-2410740. YY and WH are supported by NIH via 1R35GM146894.

\bibliographystyle{elsarticle-num}
\bibliography{references}

\end{document}